\documentclass[]{scrartcl}

\usepackage[utf8]{luainputenc}
\usepackage[USenglish]{babel}
\usepackage{csquotes}

\usepackage[a4paper,top=27mm,bottom=20mm,inner=25mm,outer=20mm]{geometry}

\usepackage[%
  backend=bibtex,bibencoding=ascii,
  style=numeric-comp,
  giveninits=true, uniquename=init, 
  natbib=true,
  url=true,
  doi=true,
  isbn=false,
  backref=false,
  maxnames=99,
  ]{biblatex}
\usepackage{amsmath}
\allowdisplaybreaks
\numberwithin{equation}{section}
\usepackage{amssymb}
\usepackage{commath}
\usepackage{mathtools}
\usepackage{bbm}
\usepackage{nicefrac}
\usepackage{subdepth}

\usepackage{siunitx}
\usepackage{algorithm}
\usepackage{algpseudocode}

\usepackage{amsthm}
\usepackage{thmtools}
\usepackage{etoolbox}
\makeatletter
\patchcmd{\thmt@setheadstyle}
 {\bgroup\thmt@space}
 {\thmt@space}
 {}{}
\patchcmd{\thmt@setheadstyle}
 {\egroup\fi}
 {\fi}
 {}{}
\makeatother
\declaretheoremstyle[
  bodyfont=\normalfont\itshape,
  headformat=\NAME\ \NUMBER\NOTE,
]{myplain}
\declaretheoremstyle[
  headformat=\NAME\ \NUMBER\NOTE,
]{mydefinition}
\newcommand{\envqed}{{\lower-0.3ex\hbox{$\triangleleft$}}}
\declaretheorem[style=myplain,numberwithin=section]{theorem}
\declaretheorem[style=myplain,numberlike=theorem]{lemma}

\declaretheorem[style=mydefinition,numberlike=theorem,qed=\envqed]{remark}

\usepackage[plainpages=false,pdfpagelabels,hidelinks,unicode]{hyperref}

\usepackage{color}
\usepackage{graphicx}
\usepackage[small]{caption}
\usepackage{subcaption}

\begingroup\expandafter\expandafter\expandafter\endgroup
\expandafter\ifx\csname pdfsuppresswarningpagegroup\endcsname\relax
\else
\fi

\usepackage{booktabs}
\usepackage{rotating}
\usepackage{multirow}

\usepackage{enumitem}

\usepackage{ifluatex}
\ifluatex
  \usepackage[no-math]{fontspec}
\else
  \usepackage[T1]{fontenc}
\fi
\usepackage{newpxtext,newpxmath}

\usepackage{xparse}

\let\epsilon\varepsilon
\let\phi\varphi
\let\rho\varrho

\usepackage{xspace}

\renewcommand{\O}{\mathcal{O}}

\newcommand{\bhat}{\widehat{b}}
\newcommand{\uhat}{\widehat{u}}

\NewDocumentCommand{\RK}{o m O{\the\numexpr#2-1\relax} m O{} O{} o}{%
  \IfValueTF{#1}{#1}{RK}%
  #2(#3)#4%
  \ifblank{#6}{}{\textsubscript{F}}%
  \ifblank{#5}{}{[#5]}%
  \IfValueT{#7}{#7}%
}

\newcommand{\dt}{\Delta t}

\renewcommand{\vec}[1]{\pmb{#1}}

\NewDocumentCommand{\opD}{m+g}{%
  \IfNoValueTF{#2}
    {D_{#1}}
    {D_{#1,#2}}%
}
\NewDocumentCommand{\opDsplit}{m+g}{%
  \IfNoValueTF{#2}
    {\widetilde{D}_{#1}}
    {\widetilde{D}_{#1,#2}}%
}
\NewDocumentCommand{\opM}{g}{%
  \IfNoValueTF{#1}
    {M}
    {M_{#1}}%
}
\NewDocumentCommand{\opQ}{g}{%
  \IfNoValueTF{#1}
    {Q}
    {Q_{#1}}%
}
\NewDocumentCommand{\opI}{g}{%
  \IfNoValueTF{#1}
    {I}
    {I_{#1}}%
}
\NewDocumentCommand{\opV}{g}{%
  \IfNoValueTF{#1}
    {V}
    {V_{#1}}%
}
\NewDocumentCommand{\opB}{g}{%
  \IfNoValueTF{#1}
    {B}
    {B_{#1}}%
}
\NewDocumentCommand{\opR}{g}{%
  \IfNoValueTF{#1}
    {R}
    {R_{#1}}%
}
\NewDocumentCommand{\opN}{m+g}{%
  \IfNoValueTF{#2}
    {N_{#1}}
    {N_{#1,#2}}%
}

\NewDocumentCommand{\fnum}{g}{%
  \IfNoValueTF{#1}
    {f^{\mathrm{num}}}
    {f^{\mathrm{num,#1}}}%
}
\NewDocumentCommand{\vecfnum}{g}{%
  \IfNoValueTF{#1}
    {\vec{f}^{\mathrm{num}}}
    {\vec{f}^{\mathrm{num,#1}}}%
}
\NewDocumentCommand{\vecfcorr}{g}{%
  \IfNoValueTF{#1}
    {\vec{f}^{\mathrm{corr}}}
    {\vec{f}^{\mathrm{corr,#1}}}%
}
\NewDocumentCommand{\fvol}{g}{%
  \IfNoValueTF{#1}
    {f^{\smash{\mathrm{vol}}}}
    {f^{\smash{\mathrm{vol,#1}}}}%
}

\newcommand{\orcid}[1]{ORCID:~\href{https://orcid.org/#1}{#1}}
\usepackage{authblk}

\newenvironment{keywords}{\par\textbf{Key words.}}{\par}
\newenvironment{AMS}{\par\textbf{AMS subject classification.}}{\par}

\title{Step size control for explicit relaxation Runge-Kutta methods preserving invariants}

\author[1]{Sebastian Bleecke}
\affil[1]{Institute of Mathematics, Johannes Gutenberg University Mainz, Germany}

\author[1]{Hendrik~Ranocha\thanks{\orcid{0000-0002-3456-2277}}}

\date{November 23, 2023} 

\makeatletter
\hypersetup{pdfauthor={Sebastian Bleecke, Hendrik Ranocha}} 
\hypersetup{pdftitle={Step size control for explicit relaxation Runge-Kutta methods preserving invariants}} 
\makeatother

\begin{document}

\maketitle

\begin{abstract}
\noindent
  Many time-dependent differential equations are equipped with invariants.
Preserving such invariants under discretization can be important, e.g.,
to improve the qualitative and quantitative properties of numerical
solutions. Recently, relaxation methods have been proposed as small
modifications of standard time integration schemes guaranteeing the
correct evolution of functionals of the solution. Here, we investigate
how to combine these relaxation techniques with efficient step size
control mechanisms based on local error estimates for explicit
Runge-Kutta methods. We demonstrate our results in several numerical
experiments including ordinary and partial differential equations.

\end{abstract}

\begin{keywords}
  explicit Runge-Kutta methods,
  step size control,
  PID controller,
  first-same-as-last (FSAL) technique,
  invariant conservation,
  relaxation
\end{keywords}

\begin{AMS}
  65L06,  
  65M20,  
  65M70   
\end{AMS}

\section{Introduction}
\label{sec:introduction}

Consider an ordinary differential equation (ODE)
\begin{equation}\label{eq:initial_val_prob}
	u'(t) = f(t, u(t)), \quad u(t^0) = u^0.
\end{equation}
We are interested in such initial value problems equipped with
an invariant $\eta$ called \emph{entropy} satisfying
$\forall t,u\colon \eta'(u) f(t,u) = 0$, i.e.,
\begin{equation}
	\frac{\dif}{\dif t} \eta(u(t)) = 0.
\end{equation}
It is often desirable to preserve the conservation of invariants
when solving the ODE numerically since this can lead both qualitative
and quantitative improvements. This line of research is well-known as
geometric numerical integration and discussed in textbooks such as
\cite{sanzserna1994numerical,hairer2006geometric}. For example, symplectic
Runge-Kutta methods are able to conserve quadratic functionals --- but are
fully implicit. In this article, we concentrate on the \emph{relaxation
method} investigated recently.
The basic ideas of relaxation methods date back to Sanz-Serna
\cite{sanzserna1982explicit} and have been refined since then, e.g.,
for inner product norms and the classical fourth-order Runge-Kutta
method by Dekker \& Verwer \cite[pp. 265-266]{dekker1984stability}
via the incremental direction technique of \cite{calvo2006preservation}
losing an order of accuracy to relaxation Runge-Kutta methods
\cite{ketcheson2019relaxation,ranocha2020relaxation}. The general theory
of \cite{ranocha2020general} broadens the scope of applications and
includes multistep methods, deferred correction and ADER schemes
\cite{abgrall2022relaxation,gaburro2023high}, IMEX methods
\cite{kang2022entropy,li2022implicit,li2023linearly}, and
multi-derivative schemes \cite{ranocha2023functional,ranocha2023multiderivative}.
Relaxation methods have also been extended to multiple invariants
\cite{biswas2023multiple,biswas2023accurate}. In this article, we consider
a single invariant $\eta$ and an explicit Runge-Kutta method given by
\cite{hairer2008solving,butcher2016numerical}
\begin{equation}
\label{eq:RK-baseline}
\begin{aligned}
	y^i &= u^n + \dt_n \sum_{j=1}^{i-1} f(t^n + c_j \dt_n, y^j),
	\\
	u^{n+1} &= u^n + \dt_n \sum_{i=1}^s b_i f(t^n + c_i \dt_n, y^i).
\end{aligned}
\end{equation}
In this case, the numerical solution $u^{n+1}$ of the baseline method
is modified to the relaxed solution
\begin{equation}
	u^{n+1}_\gamma = u^n + \gamma (u^{n+1} - u^n),
\end{equation}
where the scalar relaxation parameter $\gamma$ is determined in each
time step as the solution of the nonlinear scalar equation
\begin{equation}
	\eta(u^{n+1}_\gamma) = \eta(u^n).
\end{equation}
From here on, we continue the time integration with
$u^{n+1}_\gamma$ as approximation at the relaxed time
$t^{n+1}_\gamma = t^n + \gamma \dt_n$ instead of $u^{n+1}$
at time $t^{n+1} = t^n + \dt_n$.
The available theory yields \cite{ranocha2020relaxation,ranocha2020general}
\begin{theorem}
  Consider the relaxation procedure described above for a Runge-Kutta
  method of order $p \ge 2$ with exact value at time $t^{n}$,
  i.e., $u^{n+1} = u(t^{n+1}) + \O(\dt^{p+1})$.
  If the time step $\dt_n$ is sufficiently small and
  \begin{equation}
  \label{eq:entropy-non-degenerate}
    \eta'(u^{n+1}) \frac{u^{n+1} - u^{n}}{\| u^{n+1} - u^{n} \|}
    =
    c \dt + \O( \dt_n^{2} ),
    \quad
    \text{with } c \neq 0,
  \end{equation}
  there is a unique solution $\gamma = 1 + \O(\dt_n^{p-1})$ and the
  relaxation method satisfies
  $u^{n+1}_\gamma = u(t^{n+1}_\gamma) + \O(\dt_n^{p+1})$.
\end{theorem}
Thus, relaxation Runge-Kutta methods have at least the same order of
accuracy as the baseline methods and preserve the invariant entropy $\eta$.
The technical assumption \eqref{eq:entropy-non-degenerate} is a certain
non-degeneracy condition basically requiring that the entropy $\eta$ is
nonlinear enough, e.g., strictly convex; since relaxation methods conserve
all linear invariants, there cannot be a unique solution $\gamma$ for a
linear functional $\eta$.

To simplify the following notation, assume that the Runge-Kutta methods
satisfy the row sum condition
\begin{equation}
	c_i = \sum_{j} a_{ij}, \quad \forall i.
\end{equation}
Then, we can assume that the ODE is given as an autonomous system.

Relaxation methods have been demonstrated to be efficient and accurate
for several problems including Hamiltonian ODEs
\cite{ranocha2020relaxationHamiltonian,zhang2020highly} and PDEs
\cite{mitsotakis2021conservative,ranocha2021rate},
kinetic equations \cite{leibner2021new,pagliantini2021energy}, and
entropy-stable methods for compressible fluid flows
\cite{yan2020entropy,waruszewski2022entropy,ranocha2020fully}.
However, there is a problem when combining relaxation methods with
so-called FSAL (first-same-as-last) methods developed for error-based
step size control \cite{dormand1980family}, which can be written as
\begin{equation}
\label{eq:FSAL}
\begin{aligned}
  y^i &= u^n_\gamma + \dt_{n} \sum_{j=1}^{i-1} a_{ij} f(y^j), \\
  u^{n+1} &= u^n_\gamma + \dt_{n} \sum_{i=1}^s b_{i} f(y^i), \\
  \uhat^{n+1} &= u^n_\gamma + \dt_{n} \sum_{i=1}^s \bhat_{i} f(y^i) + \dt_{n} \bhat_{s+1} f(u^{n+1}).
\end{aligned}
\end{equation}
The embedded solution $\uhat^{n+1}$ is used to estimate the error in one
step and adapt the time step size accordingly. The inclusion of the term
$f(u^{n+1})$ enables the construction of better error estimators and does
not increase the cost of the baseline method if the step is accepted since
it has to be computed in the first stage of the next step anyway.

However, using relaxation leads to a reduced efficiency for FSAL methods
since the relaxation procedure is typically performed after computing
the embedded solution, resulting in the need to compute the right-hand
side of the relaxed solution at the beginning of the next time step
\cite{rogowski2022performance,jahdali2022performance}.
Our goal in this paper is to fix this inefficiency and develop new
approaches to combine relaxation with FSAL methods efficiently.
To do so, we first describe our new approaches in the following
Section~\ref{sec:basic-description}. Afterwards, we analyze the methods
in Section~\ref{sec:theory} and compare them in numerical experiments
in Section~\ref{sec:numerical_experiments}. Finally, we summarize and
dicsuss our results in Section~\ref{sec:summary}.

\section{Basic description of the new methods}
\label{sec:basic-description}

Before diving deeper into the theoretical details of our novel methods,
we summarize them briefly.
We start out by taking a closer look at the current procedure when working with existing software libraries.
Afterwards we motivate a modification of said procedure and present two new approaches.
Furthermore we will discuss possible obstacles and challenges along the way.

\subsection{Naive combination of step size control and relaxation}
\label{sec:naive}

To derive the new methods presented in this work, we first look at the current procedure
when using software packages such as OrdinaryDiffEq.jl \cite{rackauckas2017differentialequations} or SUNDIALS
\cite{gardner2022sundials,hindmarsh2005sundials}\footnote{%
Relaxation methods are available in recent versions of ARKODE
\cite{reynolds2023arkode} in SUNDIALS.}.
Given a baseline method, we first compute the stages
\begin{equation}
  y^i = u^n_\gamma + \dt_{n} \sum_{j=1}^{i-1} a_{ij} f(y^j),
\end{equation}
followed by the update of the main solution
\begin{equation}
  u^{n+1} = u^n_\gamma + \dt_{n} \sum_{i=1}^s b_{i} f(y^i).
\end{equation}
Next we compute the update
\begin{equation}
   \uhat^{n+1} = u^n_\gamma + \dt_{n} \sum_{i=1}^s \bhat_{i} f(y^i) + \dt_{n} \bhat_{s+1} f(u^{n+1})
\end{equation}
of the embedded method and use a PID controller
\cite{soderlind2006time,soderlind2006adaptive,kennedy2003additive}
to update the time step size $\dt$ and decide whether the step
should be accepted or rejected.

In a nutshell, a PID controller works as follows; see
\cite{soderlind2006time,soderlind2006adaptive,ranocha2023error}
for more details and extensions. Given the main and embedded updates,
we compute the weighted error estimate
\begin{equation}
	w_{n+1} = \left(
		\frac{1}{N}
		\sum_{i=1}^N \left(
			\frac{u_i^{n+1} - \uhat_i^{n+1}}
				 {\tau_a + \tau_r \max\{ |u_i^{n+1}|, |\uhat_i^{n+1}|\}}
		\right)^2
	\right)^\frac{1}{2},
\end{equation}
where $N$ is the number of components of $u$ and $\tau_a$/$\tau_r$ are
given absolute/relative tolerances. Then, we set
$\epsilon_{n+1} = 1 / w_{n+1}$ and compute the new time step size as
\begin{equation}
	\dt_{n+1}
	=
	\left(1 + \arctan\biggl(
			\epsilon_{n+1}^{\beta_1 / p}
			\epsilon_n^{\beta_2 / p}
			\epsilon_{n-1}^{\beta_3 / p} - 1
	\biggr)\right) \dt_n,
\end{equation}
using a step size limiter as suggested in \cite{soderlind2006adaptive}.
If the value multiplying $\dt_n$ is smaller than the default safety factor
$0.9^2 = 0.81$, the step is rejected an re-tried with the smaller time
step size predicted by the PID controller.

When the step is accepted, relaxation is performed by computing the
relaxation parameter $\gamma$ as root of the scalar equation
\begin{equation}
	\eta(u^{n + 1}_{\gamma}) = \eta(u^n_{\gamma}),
\end{equation}
where
\begin{equation}
	u^{n + 1}_{\gamma} = u^n_{\gamma} + \gamma (u^{n + 1} - u^n_{\gamma}).
\end{equation}
Afterwards, we adjust $u^{n + 1}$ and $t^{n + 1}$ to $u^{n + 1}_{\gamma}$ and
$t^{n + 1}_\gamma = t^n + \gamma \dt_n$, respectively.

However, we have to compute $f(u^{n + 1}_{\gamma})$ at the beginning
of the next step since we cannot reuse the value $f(u^{n+1})$ from the
FSAL stage.
In total this amounts to one additional right-hand side evaluation,
thus losing efficiency of the FSAL structure.
To keep its efficiency, we present two new approaches in the following
sections.

\subsection{FSAL-R}
To keep the efficiency of the FSAL structure, we first need to decide at which point relaxation should be performed.
Following the structure of the naive combination of FSAL and relaxation,
it seems natural to first perform step size control and
relaxation afterwards. Thus, the idea is to first compute the stages
\begin{equation}
  y^i = u^n_\gamma + \dt_{n} \sum_{j=1}^{i-1} a_{ij} f(y^j),
\end{equation}
the numerical solution
\begin{equation}
  u^{n+1} = u^n_\gamma + \dt_{n} \sum_{i=1}^s b_{i} f(y^i),
\end{equation}
and the embedded solution
\begin{equation}
   \uhat^{n+1} = u^n_\gamma + \dt_{n} \sum_{i=1}^s \bhat_{i} f(y^i) + \dt_{n} \bhat_{s+1} f(u^{n+1}).
\end{equation}
Using the same PID controller logic, we check if the step is rejected
and determine the new time step size.
Afterwards, relaxation is performed on the accepted update $u^{n + 1}$ yielding the solution
$u^{n + 1}_{\gamma}$ at time $t^{n + 1}_{\gamma}$.

Until now we followed the same procedure as the naive combination of FSAL and relaxation.
To save one additional right-hand side evaluation, we use an approximation
for the first right-hand side evaluation $f(u^{n + 1}_{\gamma})$ at the
beginning of the next step.
A simple solution is to approximate $f(u^{n + 1}_{\gamma})$ by
$f(u^{n + 1})$, thus using the old FSAL structure for this new method.
A more sophisticated way to approximate $f(u^{n + 1}_{\gamma})$ is to use
\begin{equation}
    f(u^{n+1}_\gamma) \approx f(u^{n}_\gamma) + \gamma \bigl( f(u^{n+1}) - f(u^{n}_\gamma) \bigr).
\end{equation}
This approximation can be motivated by taking the main result
\begin{equation}
	u^{n+1}_\gamma = u^{n}_\gamma + \gamma \bigl( u^{n+1} - u^{n}_\gamma \bigr)
\end{equation}
from \cite{ranocha2020general} and applying a sufficiently smooth function $f$ to both sides.
The accuracy of this approximation will be analyzed in
Section~\ref{FSAL-R_theory_section}.

The FSAL-R approach has the advantage of following the basic structure of existing
software libraries.
Additionally, relaxation is only performed if a step is accepted. Finally, the right-hand
side update $f(u^{n + 1}_{\gamma})$ is computed using only a linear combination of
values that are already computed in the process, thus saving one additional right-hand
side in comparison to the naive approach described in Section~\ref{sec:naive}.

\subsection{R-FSAL}
Besides performing relaxation after the step size control procedure, it is also
possible to perform relaxation in between the computation of a baseline step and the step size control.
We start off by computing the stages of our baseline method
\begin{equation}
  y^i = u^n_\gamma + \dt_{n} \sum_{j=1}^{i-1} a_{ij} f(y^j),
\end{equation}
followed by the update
\begin{equation}
  u^{n+1} = u^n_\gamma + \dt_{n} \sum_{i=1}^s b_{i} f(y^i).
\end{equation}
Instead of checking if the update $u^{n + 1}$ is sufficiently accurate one now performs relaxation yielding $u^{n + 1}_{\gamma}$ and $t^{n + 1}_{\gamma}$.
Afterwards, we compute the embedded method $\widehat{u}^{n + 1}$.
Since we deviated from the basic structure of the naive procedure some options present themselves when it comes to the computation of $\widehat{u}^{n + 1}$.
First we need to decide which value to substitute for the term $f(u^{n + 1})$ in
\begin{equation}
   \uhat^{n+1} = u^n_\gamma + \dt_{n} \sum_{i=1}^s \bhat_{i} f(y^i) + \dt_{n} \bhat_{s+1} f(u^{n+1}).
\end{equation}
Borrowing the approach from the FSAL-R method we can make a simple substitution by approximating $f(u^{n + 1})$ via $f(u^{n + 1}_{\gamma})$.
Note that the use of $f(u^{n + 1}_{\gamma})$ does not imply an additional right-hand side evaluation since the last stage is not used to calculate $u^{n + 1}$.
An additional option is to reverse the interpolation formula as well and
substitute
$f(u^n_\gamma) + \frac{1}{\gamma} (f(u^{n + 1}_{\gamma}) - f(u^n_\gamma))$ for $f(u^{n + 1})$.

Besides the question of approximating the last stage derivative $f(u^{n + 1})$, we are also faced
with the decision at which time step the embedded solution $\widehat{u}^{n + 1}$should be evaluated.
We can use the old time step $\Delta t_n$ or the relaxed time step $\gamma \Delta t_n$.

Finally there are some options for evaluation of the error itself, namely
\begin{enumerate}
	\item comparing $u^{n+1}$ and $\uhat^{n+1}$,

    \item comparing $u^{n+1}_\gamma$ and $\uhat^{n+1}_\gamma := u^n_\gamma + \gamma (\uhat^{n+1} - u^n_\gamma)$,

    \item comparing $u^{n+1}_\gamma$ and $\uhat^{n+1}$.
\end{enumerate}

We conclude that we keep the FSAL structure while not having to evaluate additional right-hand sides.
On the other hand, the R-FSAL method deviates from the standard procedures of existing software libraries.
Thus, an implementation requires more involved changes.
Another drawback of the R-FSAL method is that we need to perform relaxation regardless if the step is rejected or not.
However, it has the advantage that the first stage derivative is
calculated exactly, which can be helpful for dissipative problems
to get an entropy estimate \cite{ranocha2020relaxation,ranocha2020general}.

\section{Theoretical results}
\label{sec:theory}

This section covers the main theoretical results of our new methods FSAL-R and R-FSAL.
We discuss each of the options given in the previous section and investigate how it affects the accuracy.
\subsection{FSAL-R}\label{FSAL-R_theory_section}
When using the approximations described in the previous section we have to ensure
\begin{equation}
\begin{aligned}
	u(t^{n + 1}) - u^{n + 1} &= \O(\Delta t^{p + 1}),\\
	u(t^{n + 1}) - \widehat{u}^{n + 1} &= \O(\Delta t^p).
\end{aligned}
\end{equation}
In particular, we have to ensure that the approximation of the first stage
$k^1 = f(y^1) = f(u^{n}_{\gamma})$ does not reduce the accuracy.
We start off by stating the main result of this section.
\begin{theorem}\label{FSALR_main_result}
	Using the approximation
	$f(u^{n + 1}_{\gamma}) \approx f(u^n_\gamma) + \gamma (f(u^{n + 1}) - f(u^n_\gamma))$
	or the
	approximation $f(u^{n + 1}_{\gamma})  \approx f(u^{n + 1})$ ensures
	\begin{equation}
	\begin{aligned}
		u(t^{n + 1}) - u^{n + 1} &= \O(\Delta t^{p + 1})\\
		u(t^{n + 1}) - \widehat{u}^{n + 1} &= \O(\Delta t^p).
	\end{aligned}
	\end{equation}
\end{theorem}
\begin{remark}
	We will prove that the simple approximation $f(u^{n + 1}_{\gamma})  \approx f(u^{n + 1})$ is of
	order $\O(\Delta t^p)$, while the interpolation approximation $f(u^{n + 1}_{\gamma}) \approx  f(u^n) + \gamma (f(u^{n + 1}) - f(u^n))$ is of order $\O(\Delta t^{p + 1})$. This results in the interpolation approximation introducing additional errors of one order higher ($\O(\Delta t^{p + 2})$) than the error of the baseline method. In contrast the simple approximation will introduce errors of the same order ($\O(\Delta t^p)$) as the baseline method. Both results are proven below.
\end{remark}
In order to prove the main result (Theorem \ref{FSALR_main_result}),
we need some preliminary results about the accuracy of the approximations we propose. In a first step we prove
\begin{lemma}\label{lemma:naive_approx_FSAL-R}
	A locally Lipschitz continuous function $g$ satisfies
	\begin{equation}
		g(u^{n + 1}_{\gamma}) = g(u^{n + 1}) + \O(\Delta t^p).
	\end{equation}
\end{lemma}
\begin{proof}
	From the Lipschitz condition we obtain (up to a local Lipschitz constant on the right-hand side)
	\begin{equation}
	\begin{aligned}
		\|g(u^{n + 1}_{\gamma}) - g(u^{n + 1})\| \lesssim& \|u^{n + 1}_{\gamma} - u^{n + 1}\|\\
		=& \|u^n_\gamma + \gamma (u^{n + 1} - u^{n}_\gamma) - u^{n + 1}\|\\
		=& \underbrace{|\gamma - 1|}_{= \O(\Delta t^{p - 1})}
		\underbrace{\|u^{n + 1} - u^{n}_\gamma\|}_{= \O(\Delta t)} = \O(\Delta t^p),
	\end{aligned}
	\end{equation}
	which yields the claim.
\end{proof}
Next up we prove the accuracy of the interpolation approximation
$f(u^{n + 1}_{\gamma}) \approx f(u^n_\gamma) + \gamma (f(u^{n + 1}) - f(u^{n}_\gamma))$
by stating
\begin{lemma}\label{lemma:interpolation_approx_FSAL-R}
	A two times continuously differentiable function $g \in \mathcal{C}^2$ satisfies
	\begin{equation}
		g(u^{n + 1}_{\gamma}) = g(u^{n}_\gamma) + \gamma (g(u^{n + 1}) - g(u^n_\gamma))
		+ \O(\Delta t^{p + 1}).
	\end{equation}
\end{lemma}
\begin{proof}
	First we approximate the function $g$ via a first order Taylor approximation around the
	point $u^{n + 1}$; this yields
	\begin{equation}\label{eq:taylor_approx_g}
	\begin{aligned}
		g(u) = g(u^{n + 1}) + g^\prime(u^{n + 1}) (u - u^{n + 1}) + \O(\|u - u^{n + 1}\|^2).
	\end{aligned}
	\end{equation}
	Using this Taylor approximation we evaluate $g$ at the point $u^{n + 1}_{\gamma}$; this
	results in
	\begin{equation}
	\begin{aligned}
		g(u^{n + 1}_{\gamma}) =& g(u^{n + 1}) + g^\prime(u^{n + 1}) (u^{n + 1}_{\gamma} - u^{n + 1})
								+ \O(\|u^{n + 1}_{\gamma} - u^{n + 1}\|^2)\\
							  =& g(u^{n + 1}) + (\gamma - 1) g^\prime(u^{n + 1}) (u^{n + 1} - u^n_\gamma)
								+ \O(\|u^{n +1}_{\gamma} - u^{n + 1}\|^2),
	\end{aligned}
	\end{equation}
	where the error consists of terms of the order $\O(\Delta t^{2p})$ since
	\begin{equation}
		\|u^{n + 1}_{\gamma} - u^{n + 1}\| = |\gamma - 1| \|u^{n + 1}  - u^{n}_\gamma\| = \O(\Delta t^p).
	\end{equation}
	Inserting a term $\gamma g(u^{n + 1})$ leads to
	\begin{equation}
	\begin{aligned}
		g(u^{n + 1}_{\gamma}) =& g(u^{n + 1}) + (\gamma - 1) g^\prime(u^{n + 1}) (u^{n + 1} - u^n_\gamma)
								+ \gamma g(u^{n + 1}) - \gamma g(u^{n + 1})
								+ \O(\Delta t^{2p})\\
							  =& g(u^{n + 1}) + g^\prime(u^{n + 1}) (u^n_\gamma - u^{n + 1})
                              \\
                              &
							  + \gamma \left(g(u^{n + 1}) - \bigl(g(u^{n + 1}) + g^\prime(u^{n + 1})
							  (u^n_\gamma - u^{n + 1})\bigr)\right)
							  + \O(\Delta t^{2p}).
	\end{aligned}
	\end{equation}
	Reusing the Taylor approximation of $g$ we derived in equation \eqref{eq:taylor_approx_g} we observe that
	 we just have two instances of a Taylor approximation of $g(u^n)$. For further manipulation we define
	 \begin{equation}
	 	e^g_{u^n_\gamma} := g(u^{n + 1}) + g^\prime(u^{n + 1})(u^n_\gamma - u^{n + 1}) - g(u^n_\gamma)
	 \end{equation}
	 as the error of approximating $g(u^n_\gamma)$ by the Taylor approximation given by equation
	 \eqref{eq:taylor_approx_g}, which results in
	 \begin{equation}
	 \begin{aligned}
	 	g(u^{n + 1}_{\gamma}) =& g(u^{n}_\gamma) + e^g_{u^n} + \gamma \left(g(u^{n + 1}) - g(u^n_\gamma)
	 							- e^g_{u^n_\gamma}\right) + \O(\Delta t^{2p})\\
	 						  =& g(u^{n}_\gamma) + \gamma \left(g(u^{n + 1}) - g(u^n_\gamma)\right)
	 						  + (1 - \gamma) e^{g}_{u^n_\gamma} + \O(\Delta t^{2p}).
	 \end{aligned}
	 \end{equation}
	 From equation \eqref{eq:taylor_approx_g} we deduce
	 \begin{equation}
		e^{g}_{u^n_\gamma} = \O(\|u^{n}_{\gamma} - u^{n + 1}\|^2) = \O(\Delta t^2),
	 \end{equation}
	 which finally amounts to
	 \begin{equation}
	 	g(u^{n + 1}_{\gamma}) = g(u^{n}_\gamma) + \gamma \left(g(u^{n + 1}) - g(u^n_\gamma)\right)  + \O(\Delta t^{p + 1}).
	 \end{equation}
\end{proof}
We now prove the main result of this section.
\begin{proof}[Proof of Theorem \ref{FSALR_main_result}]
	We first prove the accuracy of our methods when using the simple approximation approach.
	This means we substitute $f(u^{n})$ into our first stage, where $u^n$ describes the $n$-th solution update without
	the application of relaxation.
	We will denote this approximated first stage via $k^1 = f(u^{n})$.
To describe the regular first stage we will use $k^1_{\gamma} = f(u^n_{\gamma})$, since we used relaxation to update the last solution step.
Furthermore we will use the shorthand $k^i = f(y^i)$ for the stage derivatives of the baseline method.
Finally to distinguish between all the different methods we will
write $\tilde{u}^{n +1}$ for the numerical solution obtained by
the naive combination of FSAL and relaxation as described in
Section~\ref{sec:naive}. We call the method using a simple
approximation in the first stage $u^{n + 1}$.

	We have to prove that our new method satisfies the accuracy conditions
	\begin{equation}\label{eq:proof_main_result_FSALR_conditions}
	\begin{aligned}
		u(t^{n + 1}) - u^{n + 1} &= \O(\Delta t^{p + 1})\\
		u(t^{n + 1}) - \widehat{u}^{n + 1} &= \O(\Delta t^p).
	\end{aligned}
	\end{equation}
	Due to previous work, we know that the naive combination of
	relaxation and step size control described in
	Section~\ref{sec:naive} already satisfies these two conditions,
	see \cite{ranocha2020general}.
	With this in mind, we first concentrate on the simple approximation
	$f(u^{n}_{\gamma}) \approx f(u^{n})$ and observe
	\begin{equation}
	\begin{aligned}
		u(t^{n + 1}) - u^{n + 1} =& u(t^{n + 1}) - u^n_\gamma - \Delta t  \left( b_1 k^1 +
									\sum_{i = 2}^{s} b_i k^i \right)\\
								 =& u(t^{n + 1}) - u^n_\gamma - \Delta t  \left( b_1 f(u^{n}_{\gamma})
								    + \O(\Delta t^p) + \sum_{i = 2}^{s} b_i k^i \right)\\
								 =& u(t^{n + 1}) - u^n_\gamma - \Delta t  \left( b_1 k^1_{\gamma}
								 + \sum_{i = 2}^{s} b_i k^i \right) + \O(\Delta t^{p + 1})\\
								 =& u(t^{n + 1}) - \tilde{u}^{n + 1} + \O(\Delta t^{p + 1})
								 = \O(\Delta t^{p + 1}),
	\end{aligned}
	\end{equation}
	where we have used the result of Lemma \ref{lemma:naive_approx_FSAL-R}. Furthermore the last equation
	follows from the fact that the naive combination satisfies
	\eqref{eq:proof_main_result_FSALR_conditions} by default.
	The same arguments yield for the embedded method
	\begin{equation}
	\begin{aligned}
		u(t^{n + 1}) - \widehat{u}^{n + 1} =& u(t^{n + 1}) - u^n_\gamma - \Delta t  \left( \widehat{b}_1
											  k^1 + \sum_{i = 2}^{s + 1} \widehat{b}_i k^i  \right)\\
										   =& u(t^{n + 1}) - u^n_\gamma - \Delta t  \left( \widehat{b}_1
											  k^1_{\gamma} + \O(\Delta t^p)
											  + \sum_{i = 2}^{s + 1} \widehat{b}_i k^i  \right)\\
										   =& u(t^{n + 1}) - \widehat{\tilde{u}}^{n + 1}
										      + \O(\Delta t^{p + 1}) = \O(\Delta t^{p}),
	\end{aligned}
	\end{equation}
	since the naive combination yields
	$u(t^{n + 1}) - \widehat{\tilde{u}}^{n + 1} = \O(\Delta t^{p})$ for the embedded method.

In a second step we look at the FSAL-R method when using an interpolation
approximation
$f(u^{n}_{\gamma})\approx f(u^{n - 1}_\gamma)+\gamma(f(u^{n})-f(u^{n - 1}_\gamma))$ for the
first stage.
We will denote this approximative stage with $k^1 = f(u^{n - 1}_\gamma) + \gamma (f(u^{n}) - f(u^{n - 1}_\gamma))$, while
	$k^1_{\gamma}$ will denote the unapproximated first stage again. Just as in the previous case, $u^{n + 1}$ will describe the update using the
	approximative stage $k^1$, while $\tilde{u}^{n + 1}$ will describe the update one would get when using a
	naive combination of FSAL and relaxation as described in
	Section~\ref{sec:naive}. For the main method this results in
	\begin{equation}
	\begin{aligned}
		u(t^{n + 1}) - u^{n + 1} =& u(t^{n + 1}) - u^n_\gamma - \Delta t \left( b_1 k^1 + \sum_{i = 2}^{s} b_i
								    k^i\right)\\
								 =& u(t^{n + 1}) - u^n_\gamma - \Delta t \left( b_1 f(u^{n}_{\gamma})
								    + \O(\Delta t^{p + 1}) + \sum_{i = 2}^{s} b_i k^i\right)\\
								 =& u(t^{n + 1}) - \tilde{u}^{n + 1} + \O(\Delta t^{p + 2})
								 = \O(\Delta t^{p + 1}),
	\end{aligned}
	\end{equation}
	where we have used the result of Lemma \ref{lemma:interpolation_approx_FSAL-R}.
	A similar argument yields for the embedded method
	\begin{equation}
	\begin{aligned}
		u(t^{n + 1}) - \widehat{u}^{n + 1} =& u(t^{n + 1}) - u^n_\gamma - \Delta t  \left( \widehat{b}_1
											  k^1 + \sum_{i = 2}^{s + 1} \widehat{b}_i k^i  \right)\\
										   =& u(t^{n + 1}) - u^n_\gamma - \Delta t  \left( \widehat{b}_1
											  k^1_{\gamma} + \O(\Delta t^{p + 1})
											  + \sum_{i = 2}^{s + 1} \widehat{b}_i k^i  \right)\\
										   =& u(t^{n + 1}) - \widehat{\tilde{u}}^{n + 1}_{\gamma}
										      + \O(\Delta t^{p + 2}) = \O(\Delta t^{p}),
	\end{aligned}
	\end{equation}
	thus proving the accuracy required by equation \eqref{eq:proof_main_result_FSALR_conditions}.
\end{proof}
\subsection{R-FSAL}
\label{sec:analysis-rfsal}

Apart from discussing different approximations of $f(u^{n + 1})$ in the embedded method, we also
look at different choices of the time chosen for the embedded method $\widehat{u}^{n + 1}$.
Despite the most natural option $t^{n + 1}_{\gamma} = t^n + \gamma \Delta t_n$, since $u^{n + 1}_{\gamma}$ is an
approximation at that time step, we also discuss the evaluation of the embedded method $\widehat{u}^{n + 1}$ at the time $t^{n + 1} = t^n + \Delta t_n$.
\begin{remark}
	Evaluating the embedded method $\widehat{u}^{n + 1}$ at the time $t^{n + 1}$ has the obvious drawback,
	that time dependent problems $u^\prime = f(t,u)$ are not equivalent anymore to the autonomous problem with
	$t^\prime = 1$ as an auxiliary equation.
\end{remark}

To make things more accessible we outline all the cases we will consider regarding approximations and time
step usage:

\begin{enumerate}
	\item No factor $\gamma$ for the embedded method and no interpolation for the FSAL stage, which results in
	\begin{equation}
	\begin{aligned}
		\widehat{u}^{n + 1} = u^n_{\gamma} + \Delta t \sum_{i = 1}^{s} \widehat{b}_i k^i + \Delta t
		 \widehat{b}_{s + 1} f(u^{n + 1}_{\gamma}).
	\end{aligned}
	\end{equation}

	\item No factor $\gamma$ for the embedded method, but we use the interpolation for the FSAL stage,
	which yields
	\begin{equation}
	\begin{aligned}
		\widehat{u}^{n + 1} = u^n_{\gamma} + \Delta t \sum_{i = 1}^{s} \widehat{b}_i k^i + \Delta t
		 \widehat{b}_{s + 1} \left(f(u^n_\gamma) + \frac{1}{\gamma}(f(u^{n + 1}_{\gamma}) - f(u^n_\gamma))\right).
	\end{aligned}
	\end{equation}

	\item Use the factor $\gamma$ for the embedded method, but no interpolation for the FSAL stage.
	This results in
	\begin{equation}
	\begin{aligned}
		\begin{aligned}
		\widehat{u}^{n + 1} = u^n_{\gamma} + \gamma \dt \sum_{i = 1}^{s} \widehat{b}_i k^i +
		\gamma \dt \widehat{b}_{s + 1} f(u^{n + 1}_{\gamma}).
		\end{aligned}
	\end{aligned}
	\end{equation}

	\item Use the factor $\gamma$ for the embedded method and the interpolation for the FSAL stage,
	which amounts to
	\begin{equation}
	\begin{aligned}
		\widehat{u}^{n + 1} = u^n_{\gamma} + \gamma \dt \sum_{i = 1}^{s} \widehat{b}_i k^i
		+ \gamma \dt \widehat{b}_{s + 1} \left(f(u^n_\gamma) + \frac{1}{\gamma}(f(u^{n + 1}_{\gamma}) - f(u^n_\gamma))\right).
	\end{aligned}
	\end{equation}
\end{enumerate}

In order to guarantee the accuracy of our method we have to ensure
\begin{equation}
\begin{aligned}
	u(t^{n + 1}_{\gamma}) - u^{n + 1}_{\gamma} =& \mathcal{O}(\Delta t^{p + 1}),\\
	u(t^{n + 1}_{\gamma}) - \widehat{u}^{n + 1}_{\gamma} =& \mathcal{O}(\Delta t^p),
\end{aligned}
\end{equation}
where $\widehat{u}^{n + 1}_{\gamma}$ denotes one of the variants of approximation we discussed earlier. Note that the first condition is already satisfied via the earlier results on relaxation, thus we only have to prove the second condition.
Before attending the accuracy proofs regarding the cases above we first prove a preliminary result in
\begin{lemma}\label{lemma:interpolation_formula_RFSAL}
	A two times continuously differentiable function $g \in \mathcal{C}^2$ admits to following approximation
	\begin{equation}
		g(u^{n + 1}) = g(u^n_\gamma) + \frac{1}{\gamma} \left(g(u^{n + 1}_{\gamma}) - g(u^n_\gamma)\right)
		+ \O(\Delta t^{p + 1}).
	\end{equation}
\end{lemma}
\begin{proof}
	We use the result about the interpolation formula
	\begin{equation}
		g(u^{n + 1}_{\gamma}) = g(u^n_\gamma) + \gamma \left(g(u^{n + 1}) - g(u^n_\gamma)\right) + \O(\Delta t^{p + 1})
	\end{equation}
	used in the theory section about the FSAL-R method, stated in Lemma \ref{lemma:interpolation_approx_FSAL-R}. Inserting this in the
	following expression yields
	\begin{equation}
	\begin{aligned}
		 & g(u^n_\gamma) + \frac{1}{\gamma} \left(g(u^{n + 1}_{\gamma}) - g(u^n_\gamma)\right) - g(u^{n + 1})\\
		=& g(u^n_\gamma) +  \frac{1}{\gamma} \left(g(u^n_\gamma) + \gamma \left(g(u^{n + 1}) - g(u^n_\gamma)\right)
		   + \O(\Delta t^{p + 1}) - g(u^n_\gamma)\right) - g(u^{n + 1})\\
		=& \frac{1}{\gamma} \O(\Delta t^{p + 1}).
	\end{aligned}
	\end{equation}
	Note that the function $\frac{1}{1 + x}$ has the first order Taylor approximation
	\begin{equation}
		\frac{1}{1 + x} = 1 + x + \O(x^2)
	\end{equation}
	 at $0$. Inserting $\gamma = 1 + \O(\Delta t^{p - 1})$ into this approximation results in
	 \begin{equation}
	 \frac{1}{\gamma} = 1 + \O(\Delta t^{p - 1}).
	 \end{equation}
	 Combining those two observations leads to
	 \begin{equation}
	 \begin{aligned}
	 	&g(u^n_\gamma) + \frac{1}{\gamma} \left(g(u^{n + 1}_{\gamma}) - g(u^n_\gamma)\right) - g(u^{n + 1})\\
	 	=& \frac{1}{\gamma} \O(\Delta t^{p + 1}) = (1 + \O(\Delta t^{p - 1})) \O(\Delta t^{p + 1})
	 	= \O(\Delta t^{p + 1}),
	 \end{aligned}
	 \end{equation}
	 thus proving the claim.
\end{proof}
Finally we come to the main result of this section and state
\begin{theorem}
	All of the four approximations for the embedded method
	$\widehat{u}^{n + 1}$ outlined above yield the
	approximation result
	\begin{equation}
		u(t^{n + 1}_{\gamma}) - \widehat{u}^{n + 1} = \O(\Delta t^p).
	\end{equation}
\end{theorem}

\begin{proof}
	We prove the cases separately and use again the notation
	$\widehat{\tilde{u}}^{n+1}$ to denote the embedded solution obtained
	by the naive combination of relaxation and FSAL described in
	Section~\ref{sec:naive}.
	\begin{enumerate}
		\item
		First we note that since $u$ is a $\mathcal{C}^1$ solution of our initial value problem it is locally
		Lipschitz continuous. Therefore we can approximate $u(t^{n + 1}_{\gamma})$ by $u(t^{n + 1})$ via
		\begin{equation}
			u(t^{n + 1}_{\gamma}) = u(t^{n + 1}) + \O(\Delta t^p)
		\end{equation}
		using the same arguments as in Lemma \ref{lemma:naive_approx_FSAL-R} since
		$t^{n+1}_\gamma - t^{n+1} = (\gamma - 1) (t^{n+1} - t^n_\gamma) = \O(\dt^p)$.
		The same Lemma~\ref{lemma:naive_approx_FSAL-R} yields
		\begin{equation}
		\begin{aligned}
			u(t^{n + 1}_{\gamma}) - \widehat{u}^{n + 1}
			=&
			u(t^{n + 1}_{\gamma}) - u^n_\gamma - \Delta t \left(\sum_{i = 1}^{s} \widehat{b}_i k^i + \widehat{b}_{s + 1}
			f(u^{n + 1}_\gamma)\right)\\
			=&
			u(t^{n + 1}_\gamma) - u^n_\gamma  - \Delta t \left(\sum_{i = 1}^{s} \widehat{b}_i k^i +
			\widehat{b}_{s + 1} f(u^{n + 1}) + \O(\Delta t^p)\right)\\
			=&
			u(t^{n + 1}_{\gamma}) - \widehat{\tilde{u}}^{n + 1} + \O(\Delta t^{p + 1})\\
			=& u(t^{n + 1}) - \widehat{\tilde{u}}^{n + 1} + \O(\Delta t^p) = \O(\Delta t^p),
		\end{aligned}
		\end{equation}
		where we used $u(t^{n + 1}) - \widehat{\tilde{u}}^{n + 1} = \O(\Delta t^p)$ which is just the given accuracy
		of our baseline embedded method.

		\item
		With the result from Lemma \ref{lemma:interpolation_formula_RFSAL} we obtain
		\begin{equation}
		\begin{aligned}
			&u(t^{n + 1}_{\gamma}) - \widehat{u}^{n + 1}
            \\
			=&
			u(t^{n + 1}_{\gamma}) - u^n_\gamma - \Delta t \left(\sum_{i = 1}^{s} \widehat{b}_i k^i + \widehat{b}_{s + 1}
			\left(f(u^n) + \frac{1}{\gamma}(f(u^{n + 1}_{\gamma}) - f(u^n))\right)\right)\\
			=&
			u(t^{n + 1}_{\gamma}) - u^n_\gamma -\Delta t \left(\sum_{i = 1}^{s} \widehat{b}_i k^i + \widehat{b}_{s + 1}
			f(u^{n + 1}) + \O(\Delta t^{p + 1})\right)\\
			=&
			u(t^{n + 1}) - \widehat{\tilde{u}}^{n + 1} + \O(\Delta t^p) + \O(\Delta t^{p + 2}) = \O(\Delta t^p).
		\end{aligned}
		\end{equation}

		\item
		Observe that
		\begin{equation}
		\begin{aligned}
			&u(t^{n + 1}_{\gamma}) - \widehat{u}^{n + 1}
            \\
            =& u(t^{n + 1}_\gamma) - u^n_\gamma -
			\gamma \dt \left(\sum_{i = 1}^{s} \widehat{b}_i k^i + \widehat{b}_{s + 1}
			f(u^{n + 1}_\gamma)\right)\\
			=& u(t^{n + 1}_{\gamma}) - u^n_\gamma - (\Delta t + \mathcal{O}(\Delta t^{p}))
			\left(\sum_{i = 1}^{s} \widehat{b}_i k^i + \widehat{b}_{s + 1} f(u^{n + 1})
			+ \mathcal{O}(\Delta t^p)\right)
		\end{aligned}
		\end{equation}
		holds, where we used $\gamma \dt = \bigl(1 + \O(\dt^{p-1}) \bigr) \dt = \dt + \O(\dt^p)$.
		Thus,
		\begin{equation}
		\begin{aligned}
  			u(t^{n + 1}_{\gamma}) - \widehat{u}^{n + 1}
			=& u(t^{n + 1}_{\gamma}) - u^n_\gamma - \Delta t \left( \sum_{i = 1}^{s} \widehat{b}_i k^i + \widehat{b}_{s + 1} 			f(u^{n + 1}) \right) + \mathcal{O}(\Delta t^p)\\
			=& u(t^{n + 1}_{\gamma}) - \widehat{\tilde{u}}^{n + 1} + \mathcal{O}(\Delta t^p)\\
			=& u(t^{n + 1}) - \widehat{\tilde{u}}^{n + 1} + \mathcal{O}(\Delta t^p) = \O(\Delta t^p),
		\end{aligned}
		\end{equation}
		where we have used the accuracy of our baseline embedded method
		$u(t^{n + 1}) - \widehat{\tilde{u}}^{n + 1} = \mathcal{O}(\Delta t^p)$.

		\item
		Finally we observe
		\begin{equation}
		\begin{aligned}
			&u(t^{n + 1}_{\gamma}) - \widehat{u}^{n + 1}
            \\
			=&
			u(t^{n + 1}_{\gamma}) - u^{n}_\gamma - \gamma \dt  \left(\sum_{i = 1}^{s} \widehat{b}_i k^i +
			\widehat{b}_{s + 1}	\left(f(u^n_\gamma) + \frac{1}{\gamma}(f(u^{n + 1}_{\gamma})
			- f(u^n_\gamma))\right)\right)\\
			=&
			u(t^{n + 1}_{\gamma}) - u^{n}_\gamma - \gamma \dt  \left(\sum_{i = 1}^{s} \widehat{b}_i k^i +
			\widehat{b}_{s + 1}	f(u^{n + 1}) + \O(\Delta t^{p + 1}) \right)\\
			=&
			u(t^{n + 1}) - \widehat{\tilde{u}}^{n + 1} +\O(\Delta t^{p + 2})\\
			=&
			u(t^{n + 1}) - \widehat{\tilde{u}}^{n + 1} +\O(\Delta t^{p}) = \O(\Delta t^p),
		\end{aligned}
		\end{equation}
		where the term $\O(\Delta t^p)$ comes from the accuracy of our embedded baseline method.
	\end{enumerate}
\end{proof}

\section{Numerical experiments}
\label{sec:numerical_experiments}
In this section, we study the newly developed algorithms numerically. First,
we introduce some general test problems. Next, we study the different variants
of the classes FSAL-R and R-FSAL separately to single out the most promising candidates.
Concerning the FSAL-R we want to verify that our methods converge with a
desired order since we approximate the first stage of the baseline Runge-Kutta method.
Afterwards, we compare the performance between R-FSAL, FSAL-R methods
and the naive combination of step size control and relaxation as well as the baseline
method without relaxation.
To conduct our numerical experiments, we primarily use two baseline methods.
The first method is BS3, the third-order scheme of Bogacki and Shampine \cite{bogacki1989a32}.
The second one is DP5, the fifth-order method of Dormand and Prince \cite{dormand1980family}.

We have implemented all numerical methods in Julia
\cite{bezanson2017julia}. To verify our implementation, we compared it
to the schemes available from OrdinaryDiffEq.jl
\cite{rackauckas2017differentialequations} --- which we also used to
compute reference solutions numerically if necessary.
We applied the algorithm of \cite{alefeld1995algorithm} implemented
in the Julia package Roots.jl \cite{verzani202roots} to compute the
relaxation parameter $\gamma$.
For PDE discretizations, we used the package SummationByPartsOperators.jl
\cite{ranocha2021sbp}, which calls FFTW.jl \cite{frigo2005design}
for Fourier discretizations.
We used Matplotlib \cite{hunter2007matplotlib} (wrapped in the Julia
package PyPlot.jl) to create the figures.
The source code to reproduce all numerical experiments is available
online \cite{bleecke2023stepRepro}.

\subsection{Test problems}
\subsubsection{Harmonic oscillator}
We start off with a simple example given by
\begin{equation}\label{eq:harm_osc}
\begin{aligned}
	&u_1^\prime(t) = - u_2(t),\\
	&u_2^\prime(t) = u_1(t),\\
	&u_1(0) = 1, u_2(0) = 0,
\end{aligned}
\end{equation}
which can be used to describe a simple oscillation in a physical process. Furthermore the system \eqref{eq:harm_osc} conserves the following entropy functional
\begin{equation}
	\eta(u) = \eta(u_1, u_2) = \|u\|^2 = u_1^2 + u_2^2.
\end{equation}
The analytical solution of \eqref{eq:harm_osc} is given by
\begin{equation}
\begin{aligned}
	u_1(t) = \cos(t),\\
	u_2(t) = \sin(t).
\end{aligned}
\end{equation}
The next few test problems add more complexity to the simple oscillation problem.

\subsubsection{Nonlinear oscillator}
As a next step, we consider a nonlinear problem. Here, we follow
\cite{ranocha2021strong,ketcheson2023computing} and consider
\begin{equation}
\begin{aligned}
	&\frac{\dif}{\dif t} u(t) =
	\begin{pmatrix}
		u_1(t)\\
		u_2(t)\\
	\end{pmatrix}
	= \|u\|^{-2}
	\begin{pmatrix}
	-u_2(t)\\
	u_1(t)\\
	\end{pmatrix},\\
	&u(0) =
	\begin{pmatrix}
		1\\
		0\\
	\end{pmatrix},
\end{aligned}
\end{equation}
with the entropy functional
\begin{equation}
	\eta(u) = \|u\|^2
\end{equation}
and the analytical solution
\begin{equation}
\begin{aligned}
	u_1(t) = \cos(t),\\
	u_2(t) = \sin(t),
\end{aligned}
\end{equation}
which are identical with the entropy functional and the analytical solution of the harmonic oscillator.
\subsubsection{Nonlinear pendulum}
Another way of modifying the harmonic oscillator test problem in a nonlinear way is to reverse the modelling process of the harmonic osclillator from a mechanical point of view and consider the original form of the dynamical description of an idealized pendulum leading towards the following system of ordinary differential equations
\begin{equation}
\begin{aligned}
	u_1^\prime(t) &= - \sin(u_2(t)),\\
	u_2^\prime(t) &= u_1(t).
\end{aligned}
\end{equation}
This problem has the invariant
\begin{equation}
	\eta(u) = \frac{u_1^2}{2} - \cos(u_2).
\end{equation}
While this system has an analytical solution based on elliptical
integrals, we will just calculate a reference solution with the
Verner method of order 9 \cite{verner2010numerically}.

\subsubsection{Time dependent harmonic oscillator with bounded angular velocity}
As discussed in the introduction there might be some problems for the R-FSAL procedure when considering error estimators $\widehat{u}^{n + 1}$ which are based on the time step $\Delta t$ instead of $\gamma \Delta t$,
specifically for time dependent-problems.
Therefore we modify our harmonic oscillator example by multiplying a term $\omega(t)$ to both equations of the system.
In this case we choose $\omega(t) = 1 + \frac{1}{2} \sin(t)$ as angular velocity. This results in
\begin{equation}
\begin{aligned}
	&u_1^\prime(t) = -\left(1 + \frac{1}{2} \sin(t)\right) u_2(t),\\
	&u_2^\prime(t) = \left(1 + \frac{1}{2}\sin(t)\right) u_1(t),\\
	&u_1(0) = 1, u_2(0) = 0.
\end{aligned}
\end{equation}
which still conserves the entropy
\begin{equation}
	\eta(u) = \|u\|^2
\end{equation}
while having an analytical solution of the form
\begin{equation}
\begin{aligned}
	&u_1(t) = \cos(1/2)  \cos\left(t - \frac{1}{2} \cos(t)\right) - \sin(1/2)  \sin\left(t - \frac{1}{2} \cos(t)\right),\\
	&u_2(t) = \sin(1/2)  \cos\left(t - \frac{1}{2} \cos(t)\right) + \cos(1/2)  \sin\left(t - \frac{1}{2} \cos(t)\right).
\end{aligned}
\end{equation}

\subsubsection{Conserved exponential entropy}
The oscillator problems described above have a special structure.
In particular, there is a superconvergence phenomenon when applying
relaxation \cite{ranocha2020relaxationHamiltonian}.
This leads to an approximation of order $p + 1$ with a baseline method of order $p$, if $p$ is odd.
To check the generally expected order of convergence numerically, we
consider a test problem used previously in \cite{ranocha2020fully}. The exponential entropy
\begin{equation}
\eta(u) = \eta(u_1, u_2) = \exp(u_1) + \exp(u_2)
\end{equation}
is conserved by solutions of the ODE
\begin{equation}
\begin{aligned}
	&u_1^\prime = -\exp(u_2),\\
	&u_2^\prime = \exp(u_1),\\
	&u_1(0) = 1, u_2 = \frac{1}{2},
\end{aligned}
\end{equation}
which has the analytical solution
\begin{equation}
\begin{aligned}
	&u_1(t) = \log(\exp(1) + \exp(3/2)) - \log(\exp(0.5) + \exp((\exp(0.5) + \exp(1)) t))\\
	&u_2(t) = \frac{\log((\exp((\exp(0.5) + \exp(1)) t)  (\exp(0.5) + \exp(1)))}{\exp(0.5) + \exp((\exp(0.5) + \exp(1)) t))}.
\end{aligned}
\end{equation}
\subsubsection{Linear advection}
To finish off the numerical experiments we consider some PDE examples. We start off with the linear advection equation
\begin{equation}\label{eq:lin_adv}
\begin{aligned}
	&\partial_t u + \partial_x u = 0 \text{ on } (0,L) \subset \mathbb{R},\\
	&u(0,x) = \exp\left(\sin\left(\frac{2\pi x}{L}\right)\right),
\end{aligned}
\end{equation}
with periodic boundary conditions for time $t \in [0, 100]$
with $L = 2$.
The analytical solution of the linear advection equation can be easily verified as
\begin{equation}
	u(t,x) = \exp\left(\sin\left(\frac{2 \pi}{L} (x -t)\right)\right).
\end{equation}
We consider the quadratic invariant
\begin{equation}\label{eq:invariant_lin_adv}
	\eta(u) = \frac{1}{2} \int |u|^2 = \frac{1}{2} \|u\|_{L^2},
\end{equation}
where $u$ denotes the solution of \eqref{eq:lin_adv}.
We will use a discontinuous Galerkin (DG) semidiscretization in space to reduce the PDE to a system of ODEs.
Specifically, we use the DG spectral element method (DGSEM) with Gauss-Lobatto-Legendre collocation
on 8 uniformly distributed elements with polynomials of degree 5
\cite{kopriva2009implementing}. To make the discretization energy-conservative,
we use the central numerical flux at all interfaces.
To mimic a conserved quantity for the semidiscrete ODE, we use a quadrature rule on the integral \eqref{eq:invariant_lin_adv}
where the nodes and the weights of the quadrature rule are provided by the DGSEM semidiscretization.

\subsubsection{BBM equation}
As a more complex PDE example, we consider the Benjamin-Bona-Mahony (BBM) equation \cite{benjamin1972model}
\begin{equation}
\begin{aligned}
	(I - \partial_x^2)\partial_t u(t,x) + \partial_x \frac{u(t,x)^2}{2} + \partial_x u(t,x) &= 0, &&t \in (0, T),
	x \in (x_{min}, x_{max})\\
	u(0,x) &= u^0(x), &&x \in [x_{min}, x_{max}],
\end{aligned}
\end{equation}
with periodic boundary conditions as considered in \cite{ranocha2021broad}.
We choose the specific values to be
\begin{equation}
\begin{gathered}
	x_{min} = 0,
	\quad
	x_{max} = 2,
	\quad
	c = \frac{12}{10},
	\quad
	A = 3(c - 1),
	\quad
	K = \frac{1}{2} \sqrt{1 - \frac{1}{c}},
\end{gathered}
\end{equation}
which produces
\begin{equation}
\begin{aligned}
	u(t,x) = \frac{A}{\cosh(K ( x - c \ t))^2}
\end{aligned}
\end{equation}
as an analytical solution of the PDE above.
The BBM equation does in fact conserve multiple quantities namely
\cite{olver1979euler}
\begin{equation}
\begin{aligned}
	&J_1^{BBM}(u) = \int_{x_{min}}^{x_{max}} u,\\
	&J_2^{BBM}(u) = \frac{1}{2} \int_{x_{min}}^{x_{max}} \left(u^2 + \left(\partial_x u\right)^2\right),\\
	&J_3^{BBM}(u) = \int_{x_{min}}^{x_{max}} (u + 1)^3.
\end{aligned}
\end{equation}
We will primarily focus to conserve either the quadratic or the cubic
invariant with a Fourier collocation
semidiscretization which will automatically conserve the linear
invariant as well \cite{ranocha2021broad,linders2023resolving}.

\subsection{Comparing FSAL-R estimator options}
Since we approximate the first stage of our Runge-Kuttta methods we want to confirm through
numerical experiments that for the considered test problems the FSAL-R method does
indeed converge with desired order. We only consider ODE problems in this case.

\begin{figure}[htbp]
\centering
\begin{subfigure}{0.8\textwidth}
  \centering
    \includegraphics[width=\textwidth]{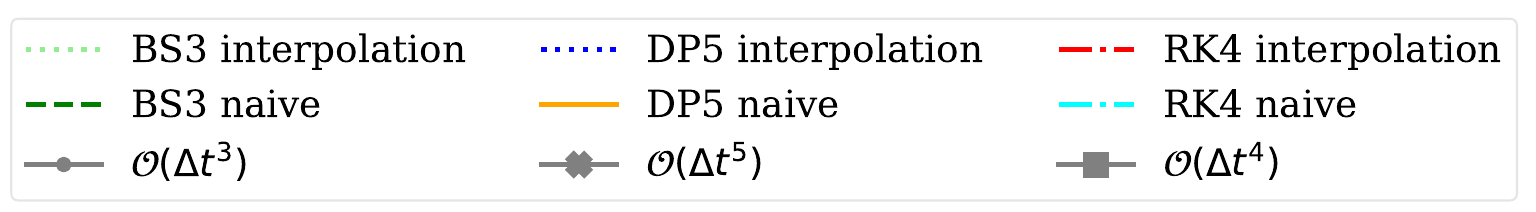}
\end{subfigure}%
\\
\begin{subfigure}{0.48\textwidth}
  \centering
  \includegraphics[width=\textwidth]{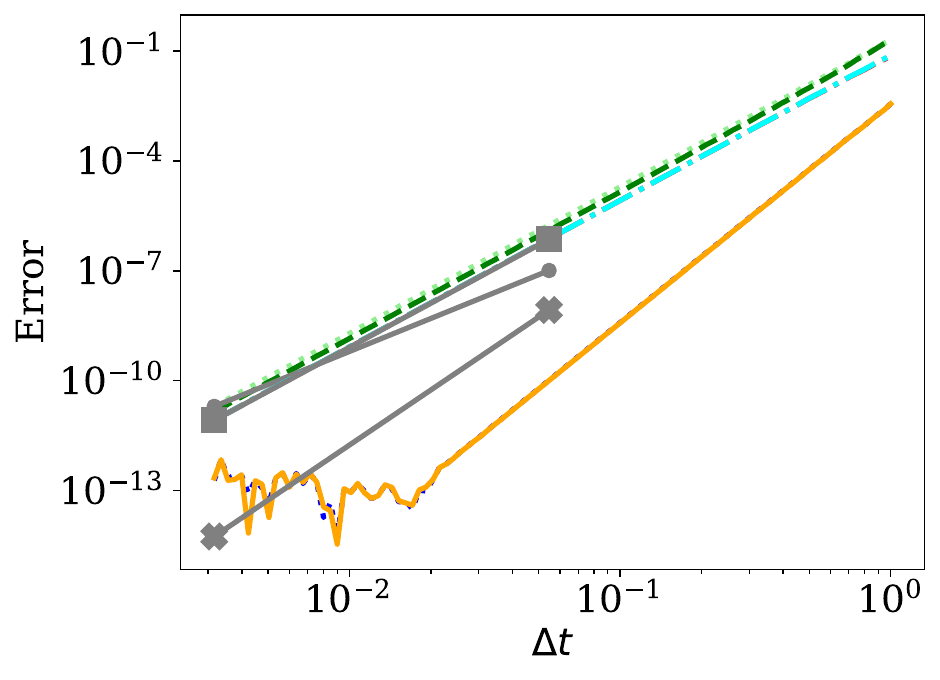}
  \caption{Harmonic oscillator.}
\end{subfigure}%
\hspace{\fill}
\begin{subfigure}{0.48\textwidth}
  \centering
  \includegraphics[width=\textwidth]{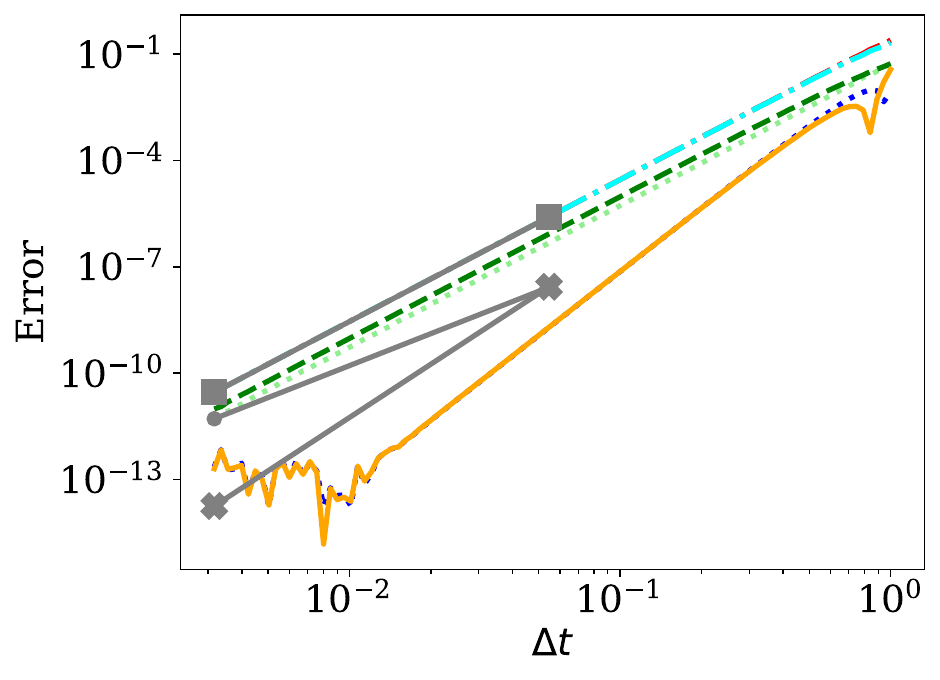}
  \caption{Non-linear oscillator.}
\end{subfigure}%
\\
\begin{subfigure}{0.48\textwidth}
  \centering
  \includegraphics[width=\textwidth]{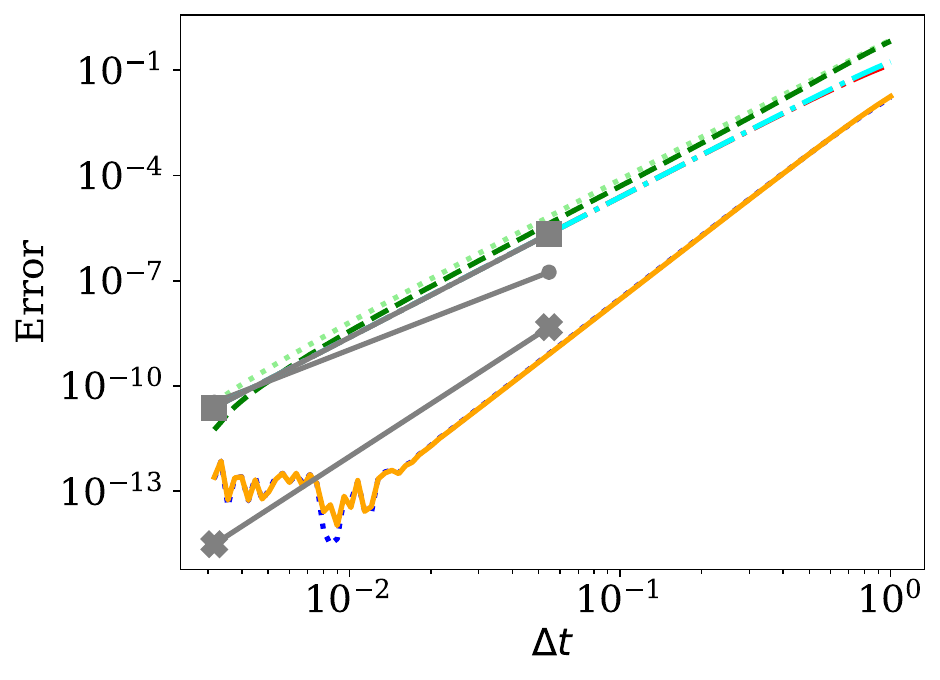}
  \caption{Bounded time dependent oscillator.}
\end{subfigure}%
\hspace{\fill}
\begin{subfigure}{0.48\textwidth}
\centering
\includegraphics[width =\textwidth]{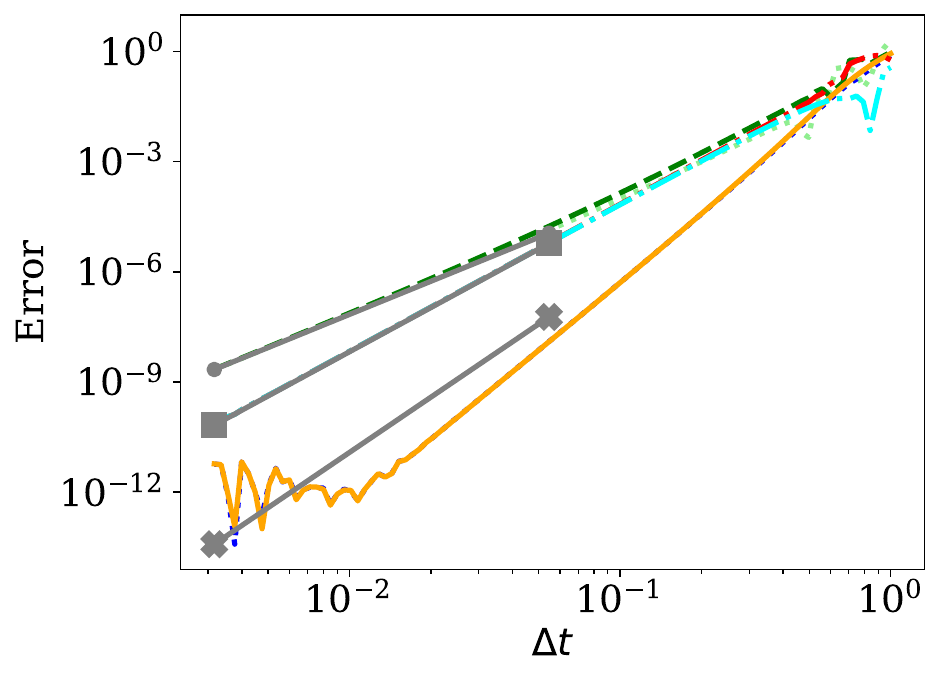}
\caption{Conserved exponential energy.}
\end{subfigure}
\caption{Convergence test for a sample of test problems using different
		 FSAL-R options.}
\label{fig:convergence_fsalr}
\end{figure}

The convergence plots are illustrated in Figure \ref{fig:convergence_fsalr}.
When considering the described oscillator problems we see that all of them show the
superconverge phenomenon regardless of the chosen algorithm BS3 or DP5,
as expected for odd-order methods \cite{ranocha2020relaxationHamiltonian}.
Thus, we also used the classical Runge-Kutta-method of order 4 \cite{kutta1901beitrag}
(short RK4) for comparison, since there is no superconvergence for this
even-order method.
However we observe that this superconvergence is not a general attribute of our FSAL-R method when using
baseline methods of odd order.
This can be seen when looking at the conserved exponential entropy test problem which
converges in the same order as both considered baseline methods.
Nevertheless we can conclude that the approximation of the first stage does not affect the convergence in any negative way.

\begin{figure}[htbp]
\centering
\begin{subfigure}{0.4\textwidth}
\centering
\includegraphics[width = \textwidth]{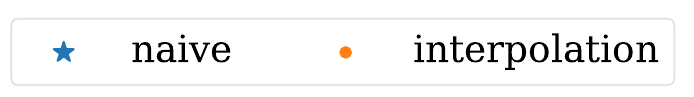}
\end{subfigure}
\\
\begin{subfigure}{0.49\textwidth}
\centering
\includegraphics[width = \textwidth]{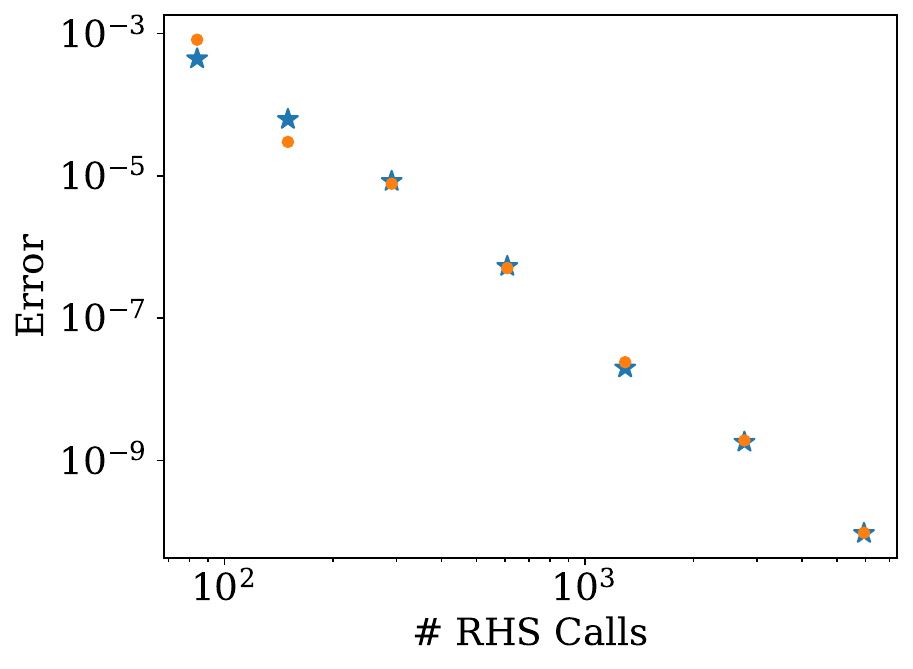}
\caption{BS3, nonlinear pendulum.}
\end{subfigure}
\hspace{\fill}
\begin{subfigure}{0.49\textwidth}
\centering
\includegraphics[width = \textwidth]{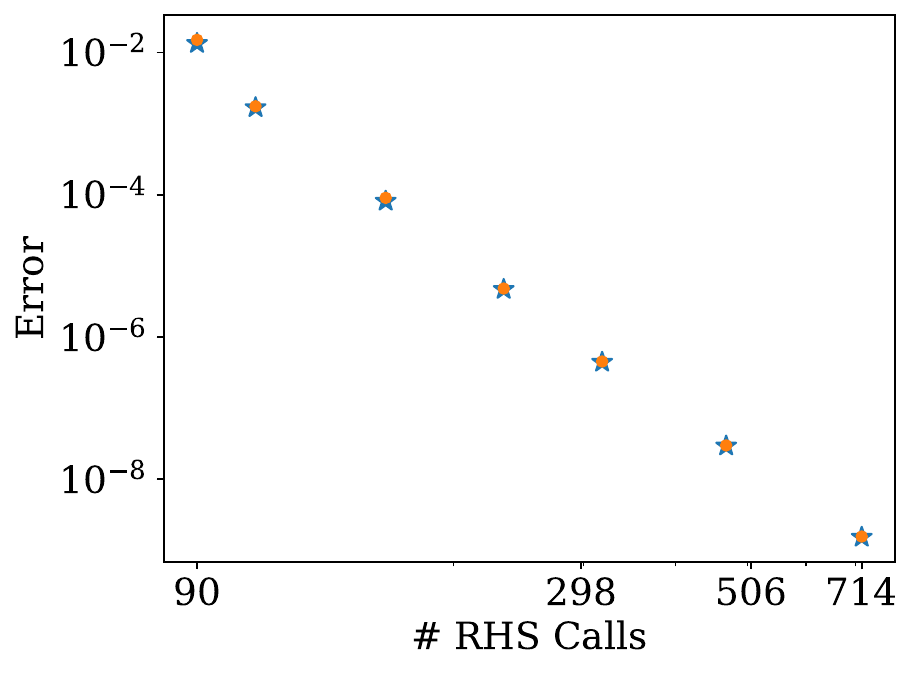}
\caption{DP5, nonlinear pendulum.}
\end{subfigure}
\\
\begin{subfigure}{0.49\textwidth}
\centering
\includegraphics[width = \textwidth]{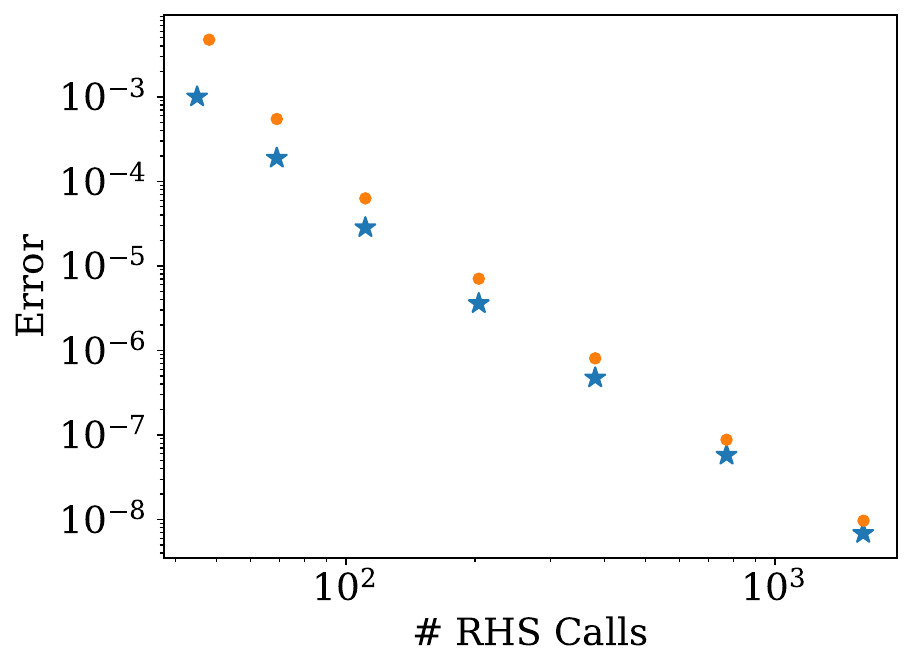}
\caption{BS3, conserved exponential entropy.}
\end{subfigure}
\hspace{\fill}
\begin{subfigure}{0.49\textwidth}
\centering
\includegraphics[width = \textwidth]{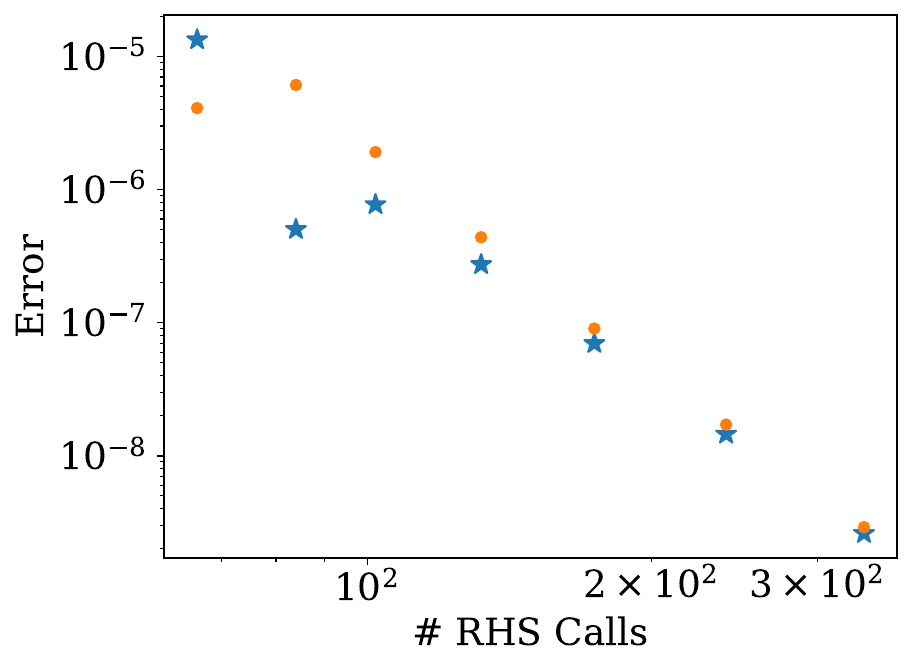}
\caption{DP5, conserved exponential entropy.}
\end{subfigure}
\\
\begin{subfigure}{0.49\textwidth}
\centering
\includegraphics[width = \textwidth]{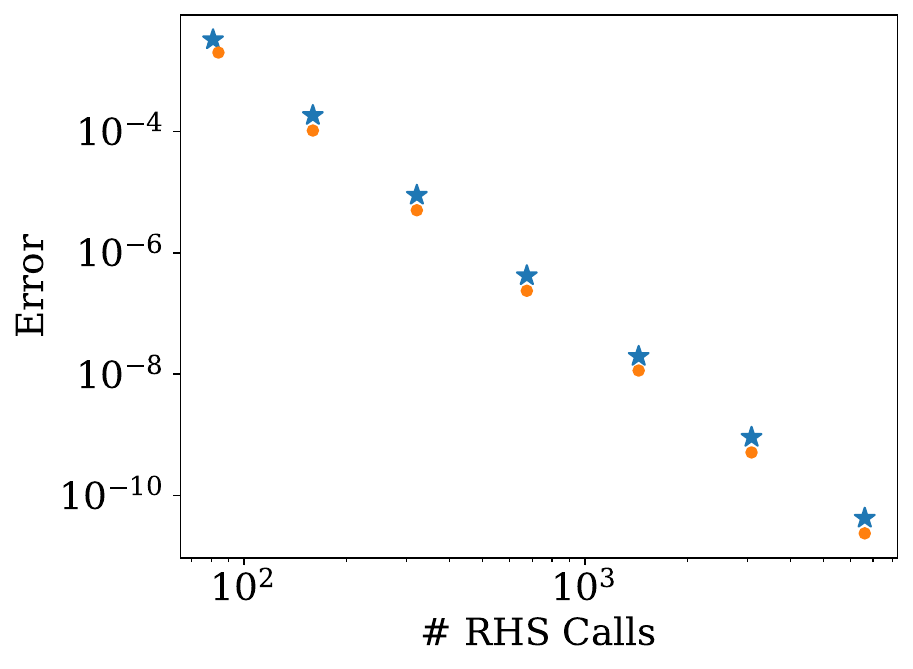}
\caption{BS3, nonlinear oscillator.}
\end{subfigure}
\hspace{\fill}
\begin{subfigure}{0.49\textwidth}
\centering
\includegraphics[width = \textwidth]{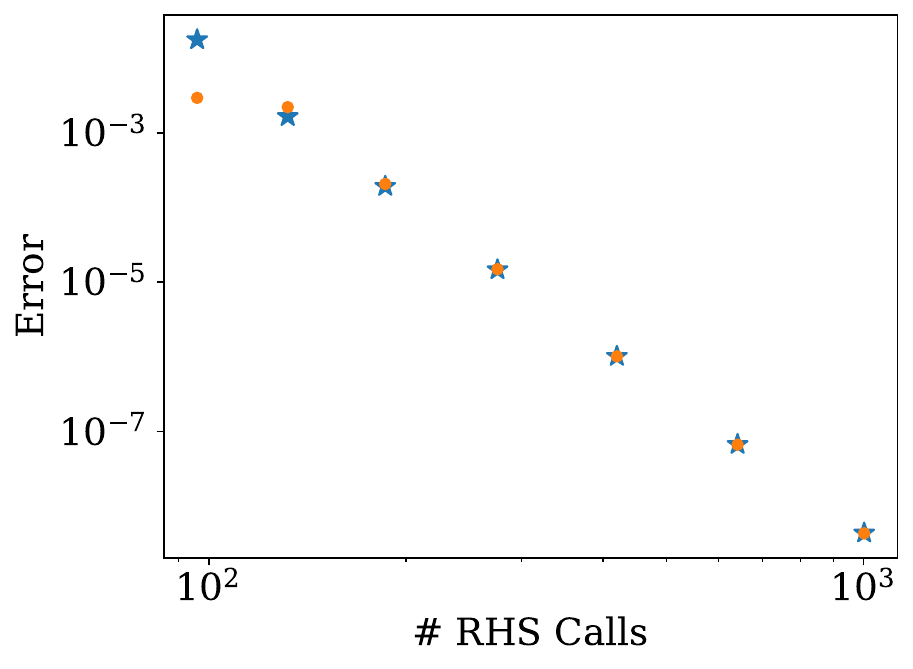}
\caption{DP5, nonlinear oscillator.}
\end{subfigure}
\caption{Work precision diagrams using different FSAL-R options.}
\label{fig:fsalr_comparison}
\end{figure}

When additionally considering step size control we compare the simple approximation with the interpolation approximation.
For the BS3 method we observe all three possible cases, which are illustrated in Figure \ref{fig:fsalr_comparison}.
The first being where there is no difference between the simple and the interpolative method as can be observed when considering  the nonlinear pendulum test problem.
When looking at the conserved exponential entropy or one of the oscillator problems we
see that the simple option is slightly better than the interpolative option when using the BS3
method.
Finally the interpolation option scores better when considering the nonlinear oscillator
problem in case of BS3.
When considering the DP5 method we see that in most cases the accuracy of both options
coincide, even if the
simple option was better than the interpolation option previously. This can be observed in
the harmonic oscillator
problems. In case of the time dependent harmonic oscillator with bounded angular
velocity  the interpolation
option even outperforms the simple option. The only test problem where the simple option
seems to perform better is the conserved exponential entropy.
The difference between the work precision diagrams regarding the BS3 and the DP5 method
can be explained with the respective order of the method.
When considering a lower order method the difference between the simple approximation of
order $\O(\Delta t^{p + 1})$ and the interpolation of order $\O(\Delta t^{p + 2})$ is
more noticeable.
Furthermore since the simple approximation has the same order of approximation as the
baseline method this can lead to a cancellation of errors resulting in a smaller global
error in total.

\subsection{Comparing R-FSAL estimator options}
When discussing the R-FSAL options of the approximations we were especially concerned with the time step $\Delta t$ instead of $\gamma \dt$ regarding the calculation of the error estimators.
When testing the $\Delta t$ options with the time dependent test problem the experiment crashes thus eliminating all the $\Delta t$ options for later comparisons with the other methods.

\begin{figure}[htbp]
\centering
\begin{subfigure}{0.8\textwidth}
\centering
\includegraphics[width = \textwidth]{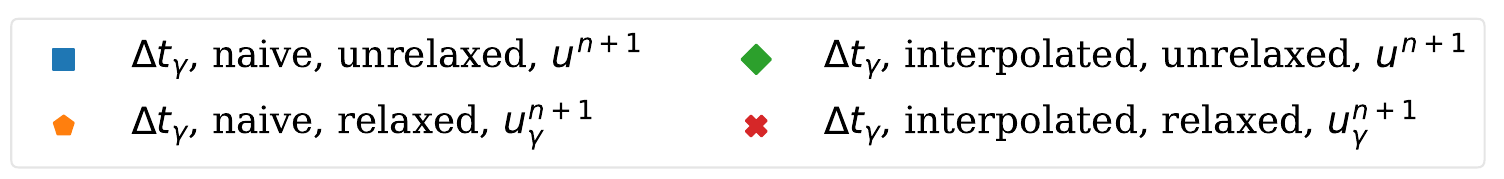}
\end{subfigure}
\\
\begin{subfigure}{0.49\textwidth}
\centering
\includegraphics[width = \textwidth]{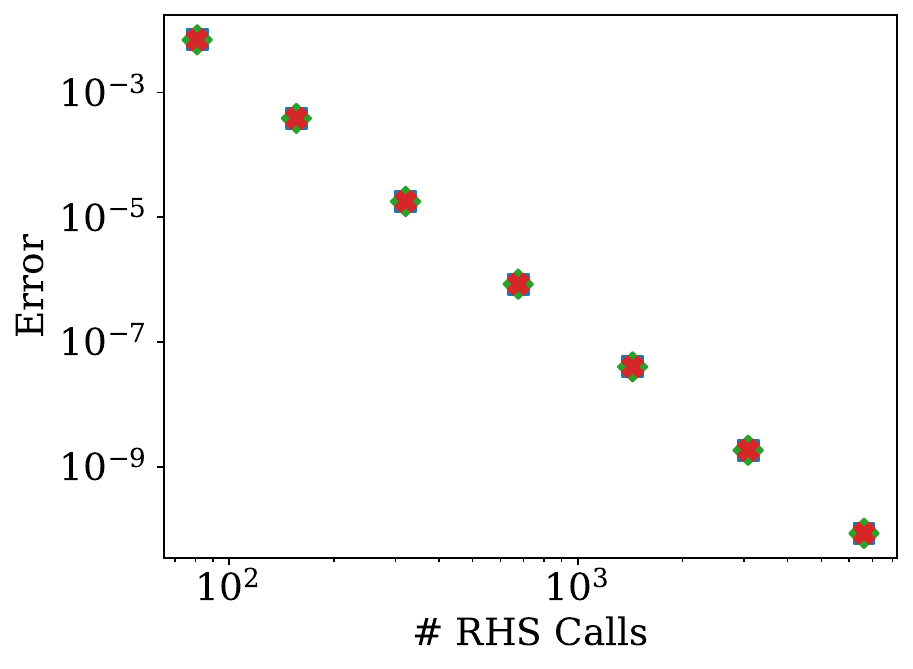}
\caption{BS3, harmonic oscillator.}
\end{subfigure}
\hspace{\fill}
\begin{subfigure}{0.49\textwidth}
\includegraphics[width = \textwidth]{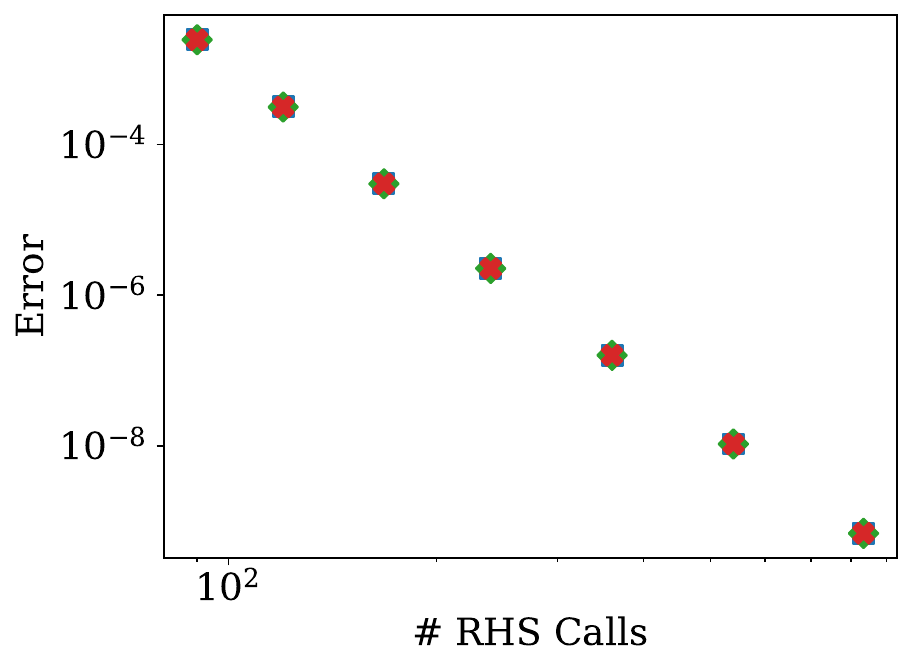}
\caption{DP5, harmonic oscillator.}
\end{subfigure}
\\
\begin{subfigure}{0.49\textwidth}
\includegraphics[width = \textwidth]{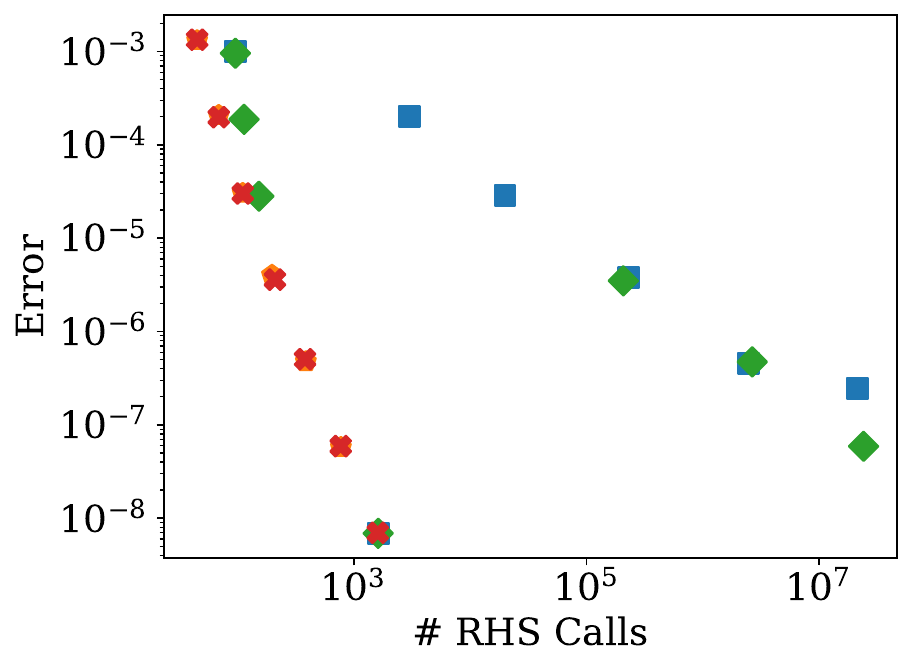}
\caption{BS3, conserved exponential entropy.}
\end{subfigure}
\hspace{\fill}
\begin{subfigure}{0.49\textwidth}
\includegraphics[width = \textwidth]{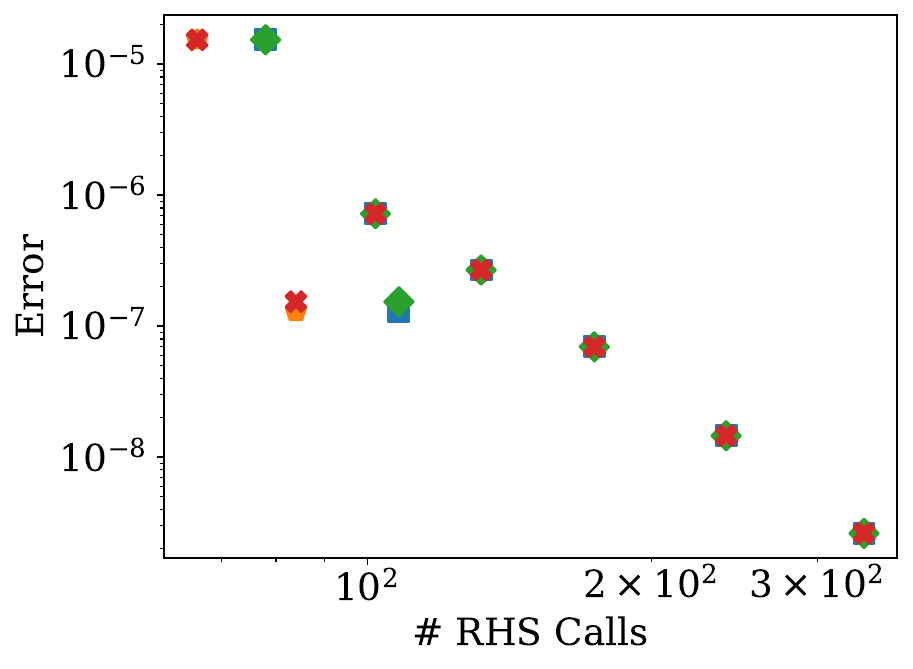}
\caption{DP5, conserved exponential entropy.}
\end{subfigure}
\caption{Work precision diagram for the R-FSAL estimator comparison.}
\label{fig:rfsal_comparison}
\end{figure}

The remaining options seem to be mostly indistinguishable for the other test problems and can be viewed in Figure \ref{fig:rfsal_comparison}. This can be observed, e.g., with the harmonic oscillator test problem.
The only test problem that stands out is the conserved exponential entropy.
In this case the error estimators using $u^{n + 1}$ seem to perform considerably worse in the BS3 case, while in the DP5 case these options seem to be doing fine.
Moving forward we will use the interpolative method with the $\gamma \dt$ time step in tandem with the value $u^{n + 1}_{\gamma}$ for the computation of the error estimator, which corresponds to the fourth option from section \ref{sec:analysis-rfsal}.
Even if there seems to be no noticeable difference when looking at the numerical experiments it still should be the safest option from a theoretical point of view.

\subsection{Step size control stability}

An important property when applying explicit time integration methods with
step size control to mildly stiff problems is step size control stability
\cite[Section~IV.2]{hairer2010solving}. This property has been first studied
by Hall \cite{hall1985equilibrium,hall1986equilibrium} with further
refinements and applications together with Higham
\cite{hall1988analysis,higham1990embedded}, leading to the development
of advanced step size control techniques such as PI or PID controllers
instead of simple I controllers \cite{gustafsson1988pi}. See also
\cite{ranocha2023error,ranocha2021optimized,ranocha2023stability}
for some recent studies in the context of computational fluid dynamics.

\begin{figure}[htpb]
\centering
\begin{subfigure}{0.7\textwidth}
\centering
\includegraphics[width = \textwidth]{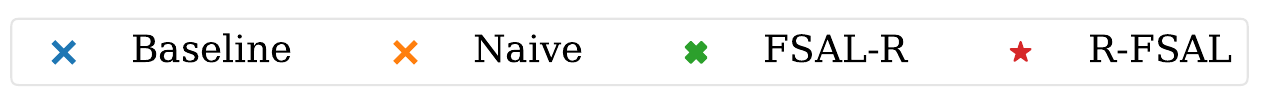}
\end{subfigure}
\\
\begin{subfigure}{0.49\textwidth}
\centering
\includegraphics[width = \textwidth]{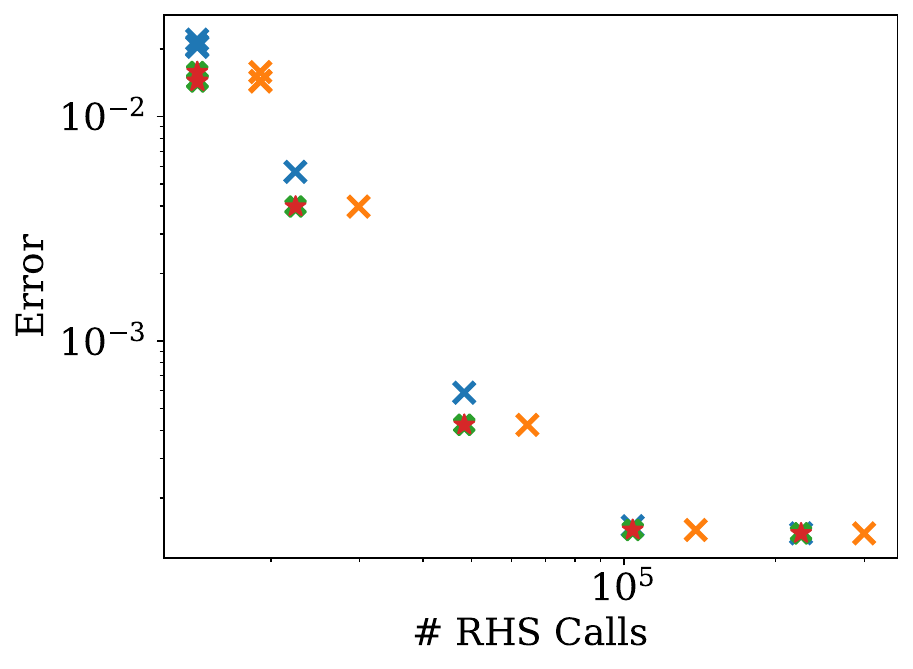}
\caption{BS3.}
\end{subfigure}%
\hspace{\fill}
\begin{subfigure}{0.49\textwidth}
\includegraphics[width = \textwidth]{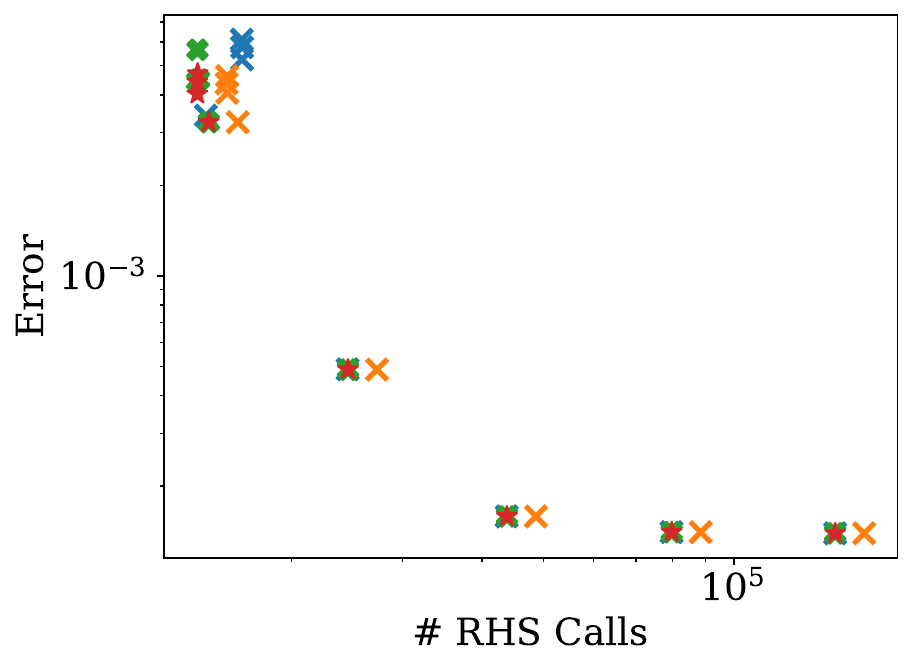}
\caption{RDPK49.}
\end{subfigure}%
\caption{Work precision diagrams for the linear advection equation
         discretized with a high-order DG method in space. The spatial
         semidiscretization is mildly stiff so that we can test the impact
         of the new methods on step size control stability.}
\label{fig:linear_advection}
\end{figure}

Here, we consider the linear advection problem. Due to the high-order DG
discretization, the problem is mildly stiff but still well-suited for explicit
time integration methods. We use BS3 since it works well for computational
fluid dynamics problems \cite{ranocha2021optimized}. We also use a fourth-order
FSAL method RDPK49 of \cite{ranocha2021optimized} that has been optimized for
such discretizations of hyperbolic PDEs.
The results are shown in Figure~\ref{fig:linear_advection}.
We can observe three different regimes. For loose tolerances, the time
step size is limited by stability. Since this combination of method and
controller is step size control stable
\cite{ranocha2021optimized,ranocha2023stability}, it works well and leads
to visually nearly indistinguishable results. For stricter tolerances, the
step size is limited by accuracy, leading to a reduction of the error combined
with an increasing number of right-hand side (RHS) calls. Since we measure the error with respect
to the PDE solution, we finally reach the regime where the error is dominated
by the spatial discretization. Thus, increasing the number of RHS calls does
not improve the error anymore.

In the regimes where the error is dominated by the error in time,
we see that relaxation reduces the error in Figure~\ref{fig:linear_advection}.
Both of our novel methods --- FSAL-R and R-FSAL --- use roughly the same
number of RHS calls as the baseline method while resulting in roughly the
same error as the naive combination of relaxation and step size control.
Thus, we do not overse any negative impact on step size control stability.

\subsection{Comparison with current methods}

For comparison we put the naive, the R-FSAL, the FSAL-R, and the baseline methods side by side in a work
precision diagram in Figure \ref{fig:main_comparison}.

\begin{figure}[htbp]
\centering
\begin{subfigure}{0.7\textwidth}
\centering
\includegraphics[width = \textwidth]{code/Plots_main/Result__BS3_bbm__work_precision_legend.pdf}
\end{subfigure}
\\
\begin{subfigure}{0.49\textwidth}
\centering
\includegraphics[width = \textwidth]{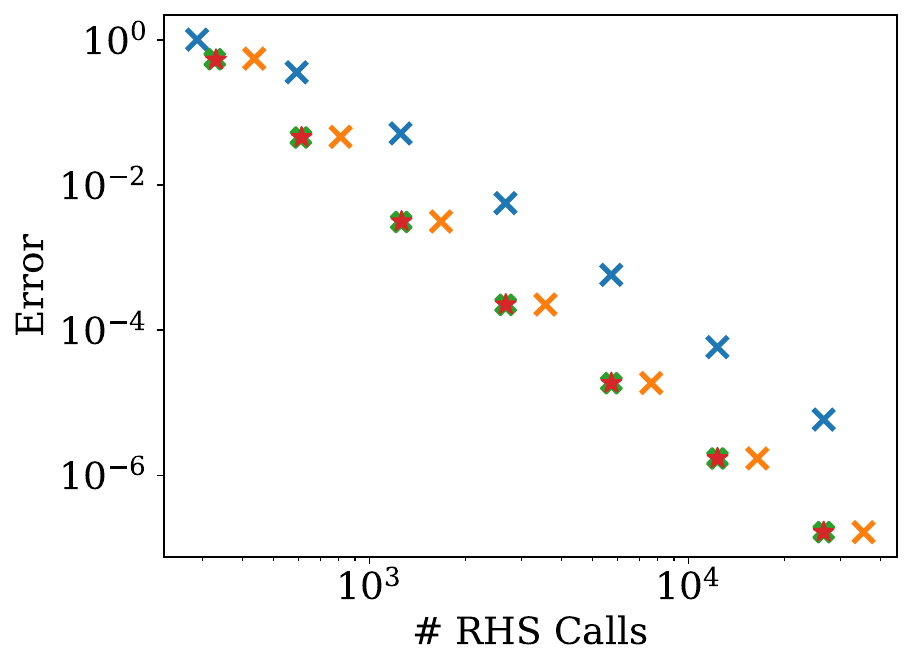}
\caption{BS3, BBM.}
\end{subfigure}
\hspace{\fill}
\begin{subfigure}{0.49\textwidth}
\centering
\includegraphics[width = \textwidth]{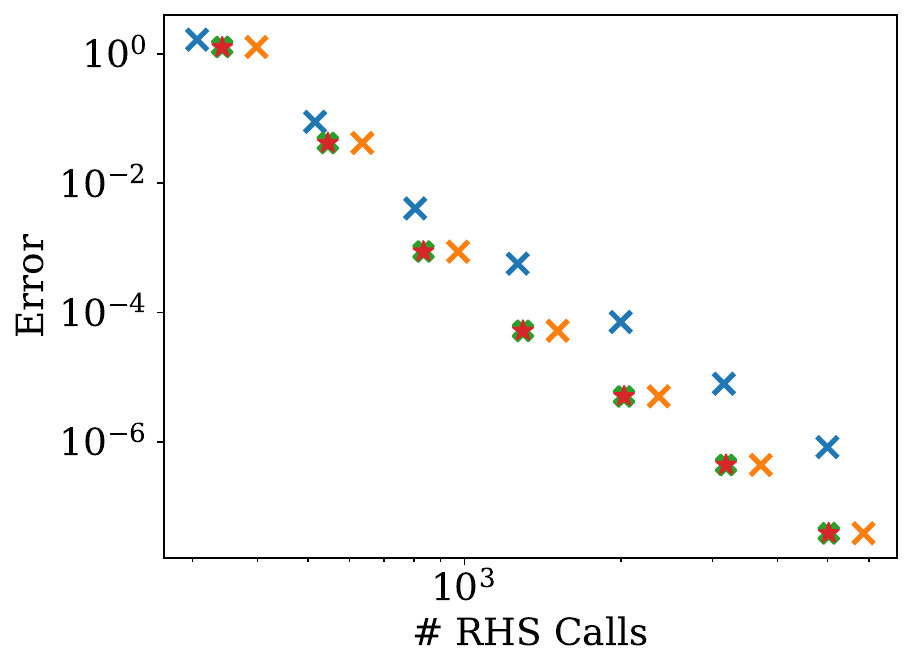}
\caption{DP5, BBM.}
\end{subfigure}
\\
\begin{subfigure}{0.49\textwidth}
\centering
\includegraphics[width = \textwidth]{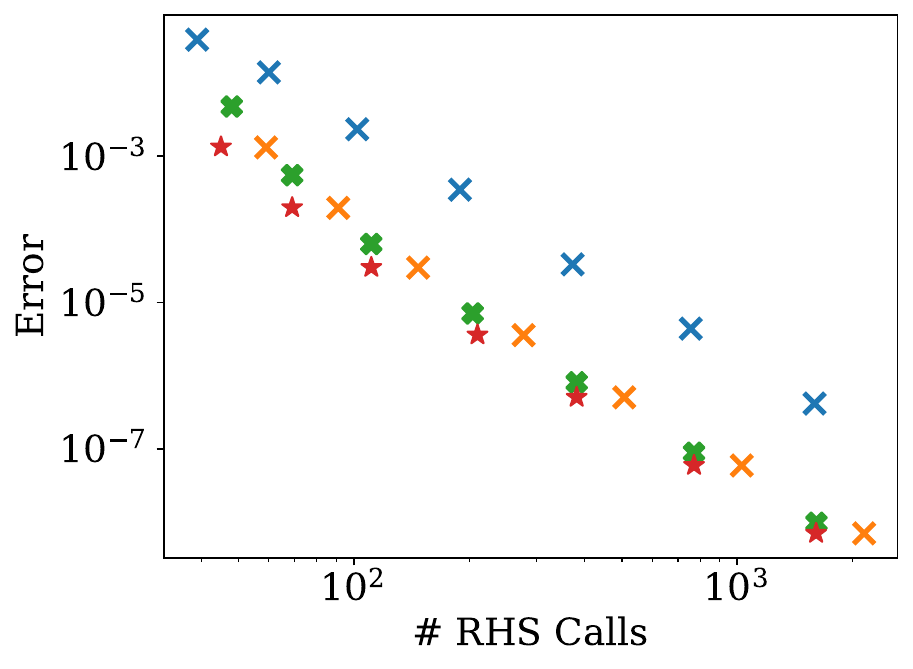}
\caption{BS3, conserved exponential entropy.}
\end{subfigure}
\hspace{\fill}
\begin{subfigure}{0.49\textwidth}
\centering
\includegraphics[width = \textwidth]{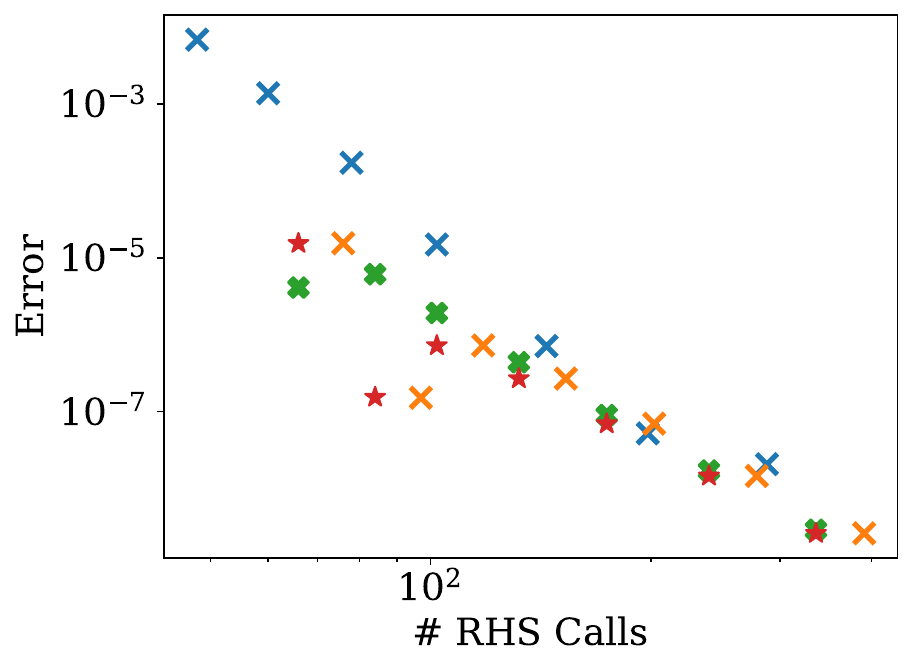}
\caption{DP5, conserved exponential entropy.}
\end{subfigure}
\\
\begin{subfigure}{0.49\textwidth}
\centering
\includegraphics[width = \textwidth]{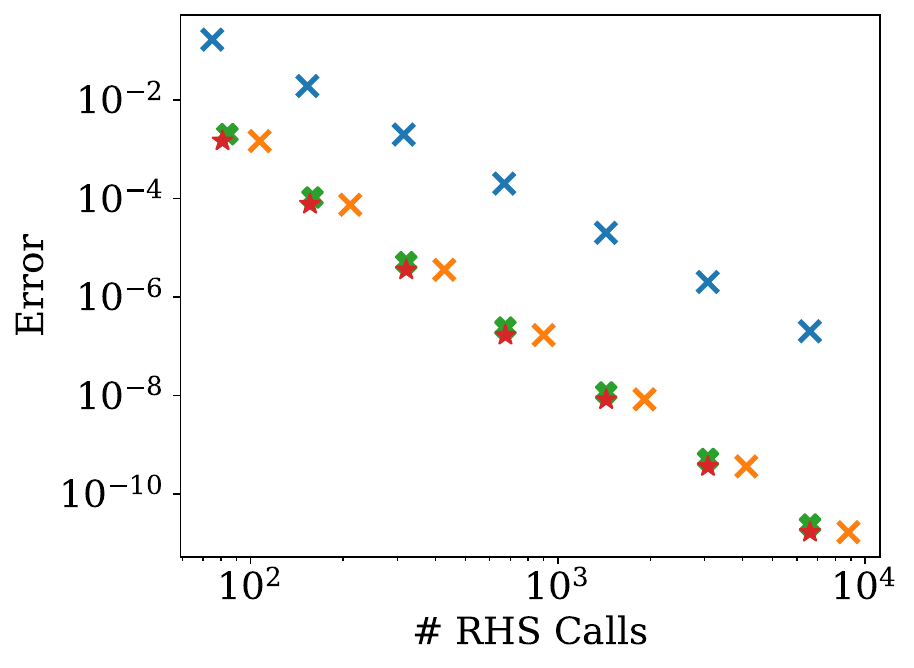}
\caption{BS3, nonlinear oscillator.}
\end{subfigure}
\hspace{\fill}
\begin{subfigure}{0.49\textwidth}
\includegraphics[width = \textwidth]{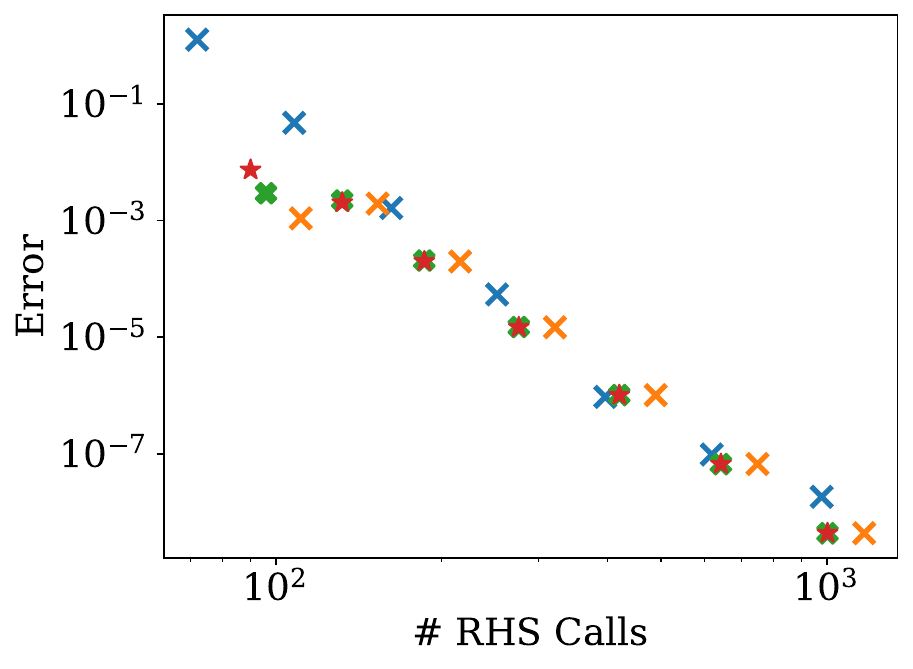}
\caption{DP5, nonlinear oscillator.}
\end{subfigure}
\caption{Work precision diagrams for the comparison of R-FSAL and FSAL-R with the existing methods.}
\label{fig:main_comparison}
\end{figure}

Most of the work precision diagrams show a typical scenario which can be seen for example in the BBM work precision diagram.
The naive method is more accurate than the baseline method but uses more right-hand side
evaluations.
When we look at the FSAL-R and R-FSAL methods they have the same accuracy as the naive method but
at the same time, due to our construction of the methods themselves, they have the same right-
hand-side evaluations as the baseline method, thus keeping the efficiency of the FSAL structure.
When we compare the R-FSAL and the FSAL-R method we see that they mostly produce
the same
accuracy for the same amount of right-hand side evaluations. When considering
BS3, we see that the R-FSAL method scores better than the FSAL-R method in case of the
conserved exponential entropy and nonlinear oscillator.

\section{Summary and conclusions}
\label{sec:summary}
In this work we investigated possible solutions for the efficient combination of relaxation and adaptive time
stepping. We developed two new methods FSAL-R, performing relaxation after step size control and R-FSAL,
performing relaxation before step size control.
In order to keep the efficiency of the baseline FSAL structure we use approximations of the right-hand side values.
In case of FSAL-R we use an approximation for the first stage $f(y^1)$.
On the other hand we use an approximation for the computation of the error estimator $\widehat{u}^{n +1}$ in
case of R-FSAL. We were able to prove that these approximations do not affect the accuracy in a negative way.
Therefore R-FSAL and FSAL-R keep the same accuracy as the naive combination of relaxation and FSAL.

When performing the numerical experiments we considered test problems with a conserved entropy functional.
The methods R-FSAL and FSAL-R are both as efficient as the baseline FSAL-method, while keeping the accuracy
of the naive combination of FSAL and relaxation.
We also observed no negative impact on step size control stability for
high-order discontinuous Galerkin discretizations of conservation laws.
Between FSAL-R and R-FSAL there seems to be no clear winner regarding accuracy and efficiency.
For a practical implementation into a software package the FSAL-R method seems to be more fitting since one does not need to adjust the core structure of the package.
However FSAL-R may run into some issues regarding dissipative problems since we use an approximation of the first stage, thus modifying the baseline method, see \cite{ranocha2020general}.

\section*{Acknowledgments}

We acknowledge support by the Deutsche Forschungsgemeinschaft
(DFG, German Research Foundation, project number 513301895)
and the Daimler und Benz Stiftung (Daimler and Benz foundation,
project number 32-10/22).

\printbibliography

\end{document}